\documentclass[a4paper,10pt]{article}

\usepackage{times}
\usepackage[T1]{fontenc}
\usepackage[latin1]{inputenc}
\usepackage{amsmath,amssymb}
\usepackage{amsthm}
\usepackage{amscd}

\newtheorem{theorem}{Theorem}[section]
\newtheorem{lemma}[theorem]{Lemma}
\newtheorem{corollary}[theorem]{Corollary}
\newtheorem{proposition}[theorem]{Proposition}

\newtheorem{definition}[theorem]{Definition}

\theoremstyle{definition}
\newtheorem{example}[theorem]{Example}
\newtheorem{remark}[theorem]{Remark}

\numberwithin{table}{section}
\numberwithin{equation}{section}

\begin{document}
\title{Simplifying operators by polynomials} 
\author{ Olavi Nevanlinna }
\maketitle

 \begin{center}
{\footnotesize\em 
Aalto University\\
Department of Mathematics and Systems Analysis\\
 email: Olavi.Nevanlinna\symbol{'100}aalto.fi\\[3pt]
}
\end{center}

\begin{abstract}
We  collect, organise known results and  add  some new  ones of the following nature:  if $A$ is  a bounded operator  in a Hilbert or Banach space,  does there exist a nonconstant polynomial $p(z)$  such that $p( A)$ is "simpler", "nicer" than $A$. 

For example $p(A)$ could be  compact  or normal   even when $A$ is not;  then one says that $A$ is polynomially compact or polynomially normal.  Using {\it multicentric calculus}  to represent scalar functions  $\varphi(z)$ as functions of $p(z)$  one can then apply functional calculus available for $p(A)$ to represent $\varphi(A)$ even so that  functional calculus  could not be formulated directly for $A$. 

 We consider  inclusion chains  of increasing generality  as for example:   finite rank $\prec$  compact $\prec$ Riesz $\prec$  almost algebraic $\prec$  quasialgebraic $\prec$  biquasitriangular $\prec$  quasiatriangular $\prec$ bounded.  We also  discuss  whether  such  classes are stable under addition, typically from a subclass.   For example,   the sum of an  almost algebraic  and  a compact operator need not be polynomially almost algebraic,  while the sum of a  polynomially almost algebraic operator  with a finite rank operator  is always polynomially almost algebraic, etc.  
 
 Block   $2 \times 2$   triangular operators are considered  as a special case, mostly in form that  if the diagonal blocks  have a property , does the whole operator share the similar polynomial property. 
\end{abstract}  
\bigskip

\bigskip

{\it Keywords:}  polynomially compact, polynomially normal, polynomially Riesz, 
  block triangular operators, multicentric calculus, functional calculus

\smallskip

MSC (2020):   47- 02, 47A55, 47A60,  47B99 

\newpage

{\bf Preface}

\bigskip

This is an attempt to collect  and organise  results on classes of bounded operators for which functional calculus based on {\it multicentric calculus} would be particularly effective.   So, what is this I  begin to call multicentric calculus?

Some dozen years ago  I decided to have a  look at  representing functions in the complex plane using "several centers"  rather than just one, the origin.  Such a representation should be useful e.g. in functional calculus when the spectrum of an operator  has a few condensed clusters.   

If $\Lambda = \{ \lambda_i\}_1^d $ denotes the centers,  we take the {\it geometric average}   of $|z-\lambda_i|$  to  measure the closeness of  $z$ to $\Lambda$:
$$
{\rm dist}(z, \Lambda )=  \prod_1^d |z-\lambda_i| ^{1/d}.
$$
This immediately  suggests to  consider a "change of variable" $w= p(z)$  with $p(z) = \prod_1^d (z-\lambda_i)$, as then sets satisfying $|p(z)| \le \rho$  are mapped to discs $|w| \le  \rho$.  In fact, already Jacobi\footnote{ C. G. J. Jacobi, \"Uber Reihenentwicklungen, welche nach den Potenzen eines gegebenen Polynoms fortschreiten, und zu Coeffizienten Polynome eines niedereren Grades haben, J. Reine Angew. Math. 53 (1856), 103 - 126 } had considered expanding functions in the form
$$
\varphi (z) \sim \sum _{n=0}^\infty c_n(z) p(z)^n
$$
where  the coefficients $c_n$  were polynomials of degree less than that of $p$.  The idea that there would be a function of a new variable $w$, for which the power series would converge in a disc was not visible, nor in the later works on Jacobi series of this type.  

What I wanted was to have $w=p(z)$ as a new {\it global} variable and  for that purpose a scalar function
$$\varphi:  z \mapsto \mathbb \varphi(z) \in \mathbb C$$
 is represented with a vector valued function 
 $$f : w \mapsto  f(w) \in \mathbb C ^d.$$ 
 One can view this point of departure  as just a simple rearrangement of terms in the expansions above, but the  scenery changes:  there is, to put it mildly,  a lot of mathematics on functions defined in  discs.    Create $f$ from $\varphi$, apply analysis to $f$ on a disc and transform the results back to  original domain and to $\varphi$. For example, $f$ is holomorphic  if and only if $\varphi$ is;   for continuous $f$ one can define a product and Banach algebra such that $\varphi$ shows up as the Gelfand transform of $f$. 
   
 In this survey the intention is to collect results  on operator classes  in which $B=p(A)$ would have a property allowing $f(B)$ to be defined using some functional calculus, although $\varphi(A)$ could  not be directly defined or composed.  The results of this nature, either deep or some of them really simple,  are scattered in the literature and this attempt most likely is far from comprehensive.
   
\bigskip

Part of the purpose of making this available is to prompt comments and pointers to known results which I have missed or overlooked.  

\bigskip

 Karjalohja  26.3.2022
 \bigskip
 
 Olavi Nevanlinna
 
 \bigskip

 \newpage

\noindent{\bf Content}

\bigskip

\noindent{1. Introduction}
\bigskip

 {1.1 Multicentric calculus: polynomials as new variables}

{1.2 Summary of inclusion relations}
\bigskip

\noindent{2. Basic   operator classes}

\smallskip

2.1 {Algebraic, almost algebraic, quasialgebraic operators}

2.2 {Excursion into  meromorphic operator valued functions}

2.3 Compact and Riesz operators

2.4 Observations on the polynomial classes

\bigskip

\noindent{3.  Special results in Hilbert spaces}

\smallskip

3.1 Special classes

3.2 Polynomially normal and polynomially unitary operators

\bigskip

\noindent{4.  Block triangular operators}

\smallskip

4.1 Notation and spectrum

4.2 Perturbation results

4.3 $M_C$ polynomially almost algebraic, compact  and Riesz

\bigskip

\bigskip

\noindent{References}

\bigskip

\noindent{Notation and definitions}

\newpage

\newpage

\section{Introduction}

\smallskip

\subsection{Multicentric calculus:  polynomials as new variables}
 
 In this paper we consider  the following problem:  given a bounded operator or matrix $A$,   does there exist a (monic) polynomial $p$ , possibly of low order, such that  $p(A)$ would be "nicer",  for example  small in norm,  diagonalizable,  compact or normal etc,  such that a suitable functional calculus could be applied to it even if it would not be available for $A$ directly.  
 
 In {\it multicentric calculus} [37]  we  represent  scalar functions  $\varphi:  \ z \mapsto \varphi(z) \in \mathbb C$ using  functions $f : \ w \mapsto f(w) \in \mathbb C^d$. 
 Without going into details how $f$ is created from $\varphi$,  if
 a scalar function $\varphi$ is given  for which $\varphi(A)$ should be defined, we may  represent $\varphi$ in {\it multicentric} form 
 \begin{equation}\label{multi}
\varphi(z) = \sum_{j=1}^d \delta_j(z) f_j(p(z))
\end{equation}
and then obtain
 \begin{equation}\label{multi}
\varphi(A) = \sum_{j=1}^d \delta_j(A) f_j(p(A)).
\end{equation}
Here $\delta_j$ is the Lagrange polynomial   $\delta_j(z) = \prod_{k\not=j}\frac{z-\lambda_k}{\lambda_j-\lambda_k}$ and hence  $\delta_j(A)$ is always well defined for any bounded operator,  while  the terms $f_j(p(A))$  would be defined and computed with suitable functional calculus.  The approach was introduced in [37] with  roots in [36]  and further developments in [38], [4], [39], [40], [2]  and  [3].  
\bigskip

We shall ask questions such as whether there exists a polynomial $p$ such that  $p(A)$ becomes compact, normal, unitary, etc.  Much of the answers are known but scattered in the literature. Perhaps the most obvious one is that of  an operator being algebraic:  if there exists a nontrivial polynomial such that $p(A)=0$, then $A$ is algebraic and  its minimal polynomial is the unique monic polynomial of smallest degree at which happens.  We  begin with listing  related definitions,  survey what is known and formulate some new  results and pose some questions.  In particular we consider block operators of the form
\begin{equation}
M= \begin{pmatrix} A & C \\
& B\end{pmatrix}
\end{equation} 
which provide a rich class of operators at which we can  demonstrate the concepts. For example,  if $A$ and $B$ share a property, does there exist a polynomial such that $p(M)$ would have that property, too. 

\subsection{Summary of inclusion relations}

We sum here up  some of the basic inclusions and related perturbations. The notation  is explained in the next section\footnote{ and  there is a list of symbols at the end of the paper}.  Many of these are obvious, most of these are known but we find it useful to collect  and present them all here for easy consideration, while the next section contains  more details and proofs or references when needed.

The basic chains are as follows
\begin{equation}
\mathcal F \subset \mathcal A \subset \mathcal A \mathcal A \subset \mathcal P \mathcal A \mathcal A  \subset \mathcal Q\mathcal A  \subset \mathcal B,
\end{equation}
\begin{equation}
\mathcal F \subset \mathcal K \subset \mathcal R  \subset \mathcal A \mathcal A \end{equation}
In addition, we have e.g.  $\mathcal N \subset \mathcal A  $ and $\mathcal N \subset \mathcal Q\mathcal N \subset \mathcal R$.    All these inclusions hold in all Banach spaces and there are separable Hilbert spaces where all inclusions are proper.  On the other hand, when we write for example
$\mathcal R \not\subset \mathcal P \mathcal K$ we mean that there is a Banach space $X$ and an operator $R \in \mathcal R(X)$ such that  for all polynomials $p$ and compact operators $K \in \mathcal K (X)$  there holds $R\not=p(K)$. 

Consider next the "stability" of these classes under summation. That is, we ask whether the sum of  any two operators from these classes  belongs to the class.
\begin{equation}
\mathcal Q\mathcal A + \mathcal K = \mathcal Q\mathcal A \ \  \text{ while } \ \ \mathcal N + \mathcal N \not\subset \mathcal Q \mathcal A 
\end{equation}
\begin{equation}
\mathcal A \mathcal A+ \mathcal F=\mathcal A \mathcal A, \ \ \mathcal P\mathcal A \mathcal A + \mathcal F = \mathcal P\mathcal A \mathcal A 
\ \  \text{ while } \ \ \mathcal A \mathcal A+ \mathcal K \not\subset \mathcal P \mathcal A \mathcal A 
\end{equation}
\begin{equation} 
\mathcal A + \mathcal F = \mathcal A,  \ \ \mathcal A + \mathcal K \subset \mathcal P  \mathcal K \ \  \text{ while } \ \  \mathcal A + \mathcal K \not\subset \mathcal A \mathcal A
\end{equation}
\begin{equation} 
\mathcal R+ \mathcal K = \mathcal R, \ \ \mathcal P\mathcal R + \mathcal K = \mathcal P\mathcal R \ \  \text{ while } \mathcal R \not\subset \mathcal P \mathcal K.
\end{equation}
Some of  these claims take a different form if we  only allow sums between commuting operators. However, we  shall not discuss that here.  

In separable Hilbert spaces   one can formulate further classes, in particular by combining properties of the operator and its adjoint together.  In particular we have

\begin{equation} 
\mathcal Q \mathcal A  \subset \mathcal B i \mathcal Q \mathcal T \subset \mathcal Q \mathcal T \subset \mathcal B,
\end{equation}
and
\begin{equation}
\mathcal N_{orm} \subset \mathcal N_{orm} + \mathcal K \subset \mathcal Q\mathcal D 
 \subset \mathcal B i \mathcal Q \mathcal T.
\end{equation}
One can also ask how an arbitrary operator in a class can be approximated using  operators which are lower in the chain.  For example ${\rm cl}\mathcal F = \mathcal K $ and 

\begin{equation}
   {\rm cl}\ {\mathcal  N} \subset {\rm cl}\ ({\mathcal  N + \mathcal K})
\subset {\rm cl}\mathcal A  
 =  \mathcal B i \mathcal Q\mathcal T.
\end{equation} 
Here again both inclusions are proper.   
All these claims are discussed below, definitions  and references given.



\section{ Basic operator classes}

\subsection{Algebraic, almost algebraic, quasialgebraic operators}

We consider only bounded operators, and write for example  $A\in \mathcal B(X)$ for bounded operators in a Banach space $X$.  

\begin{definition} If $A\in \mathcal B(X)$  is such that there exists a nontrivial polynomial $p$  such that $p(A)$ has a property $\mathcal X$, then we say that $A$ is polynomially $\mathcal X$ and denote it by $A\in \mathcal P \mathcal X$.  We say that $A$ is polynomially $\mathcal X$ of degree $d$ if  $d$ is the smallest degree of such a polynomial.
\end{definition}
\begin{remark}   Some authors have denoted $A \in {\bf Poly}^{-1} (\mathcal X)$ for the same purpose.
\end{remark}
\begin{remark}{\it Warning.}  A similar expression is sometimes used also in the case where $A$ has a property and this property is shared with $p(A)$ for all polynomials. Notice that these are very different concepts. For example, there are hyponormal operators $A$ such that $A^2$ is not hyponormal, and $A$ is called  {\it polynomially hyponormal} if   $p(A)$ is hyponormal for all  polynomials $p$, [12].

\end{remark}

\begin{example}  In infinite dimensional spaces the identity operator $I$ is not compact.  But with $p(z)=z-1$ we  have $p(I)=0$ and thus $I$ is polynomially compact.  More generally,  all algebraic operators are polynomially compact as they  vanish at  their characteristic polynomials. 
\end{example}

 \begin{example}   The $0$ operator with $p(z) = z q(z) +1$ satisfies
 $p(0)= I$ so it would be polynomially unitary.  However,  in order to make the concept useful we shall require in the unitary case that  $p$ is of the form $p(z)=zq(z)$.  Thus, invertible  algebraic operators are polynomially unitary. In fact, the inverse  of $A$ is a polynomial $q(A)$ and we have
 $A q(A)=I$ which is unitary.
 \end{example}  

{\bf Notation}   Let $X$ be a Banach space and $H$ likewise a Hilbert space.  Then 
we denote by $\mathcal F(X)$  operators with finite rank, by $\mathcal K (X)$ the {\it compact} operators, by $\mathcal A (X)$ the  {\it algebraic} operators, by $\mathcal Q \mathcal A (X) $ the {\it quasialgebraic } operators and  by $\mathcal A \mathcal A \mathcal  (X)$ the  {\it almost algebraic} ones. With $\mathcal N(X)$ we denote the nilpotent operators, with $\mathcal Q\mathcal N(X)$ the quasinilpotent ones.   In Hilbert spaces we have $\mathcal N_{orm}$ the {\it normal} operators, $\mathcal U$ for unitary ones and $\mathcal Q\mathcal T$ for quasitriangular ones.    When the space is clear from the context, we simply write $\mathcal B, \mathcal K, \mathcal A$ etc.    With $H=\mathbb C^n$ we  write $\mathcal  B(H)=\mathbb M_n(\mathbb C)$.

\bigskip
Recall the following definitions.
\begin{definition} A bounded operator $A$ is algebraic, $A \in \mathcal A$,  if there exists a  nontrivial polynomial $p$ such that  $p(A)=0$.    

It  is almost algebraic, $A\in \mathcal A \mathcal A$, if there exists a sequence $\{a_j\}$ of complex numbers such that if we put 
\begin{equation}
p_j(z)=z^j + a_1 z^{j-1}+ {\cdots} + a_j
\end{equation}
then as $j\rightarrow \infty$
\begin{equation}
\| p_j(A)\|^{1/j} \rightarrow 0.
\end{equation}
Finally,  $A$ is quasialgebraic, $A \in \mathcal Q\mathcal A$, if  
\begin{equation}
\inf \|p(A)\|^{1/{\rm deg}(p) } =0,
\end{equation}
where the infimum is over all monic polynomials.
\end{definition}

\begin{proposition}
We have in all Banach spaces inclusions
\begin{equation}\label{ekainkluusioketju}
\mathcal F \subset \mathcal A \subset \mathcal A \mathcal A \subset \mathcal Q \mathcal A \subset \mathcal B.
\end{equation}
All inclusions are proper in separable infinite dimensional  Hilbert spaces.
\end{proposition}

\begin{proof}
When $X$ is finite dimensional $\mathcal F = \mathcal B$ while  for example in infinite dimensional separable Hilbert spaces all inclusions are proper.  In fact, in  infinite dimensional spaces the identity $I$ is algebraic but not  of finite rank.  If $ \{e_j\}$ is a sequence of orthonormal vectors and $\{ \lambda_j\}$ a sequence of nonzero complex numbers converging to $0$
then  the associated diagonal operator $K: e_j \mapsto \lambda_j e_j$ is compact  and thus in $\mathcal A\mathcal A$ but is not algebraic.  As $K$ is almost algebraic it is also quasialgebraic and  so is automatically also $1+K$, which however, is not almost algebraic. Finally, for example the  forward shift $e_j \mapsto e_{j+1}$  has large spectrum and is not quasialgberaic as the following theorem shows.
\end{proof}
By a theorem by Halmos, $A$ is quasialgebraic if and only if the logarithmic capacity of its spectrum vanishes [23].   Let ${\rm cap}(\sigma(A))$ denote the logarithmic capacity of the spectrum.  


\begin{theorem} {\rm (P. Halmos  [23])} If $A\in \mathcal  B(X)$, then
\begin{equation}
\inf \|p(A)\|^{1/{\rm deg}(p) } =  {\rm cap}(\sigma(A)),
\end{equation}
where the infimum is over all monic polynomials $p$.
\end{theorem}

If $A\in \mathcal  B(X)$ can be approximated fast enough with algebraic operators, then $A$ is always quasialgberaic.

\begin{definition}  Given $A\in \mathcal  B(X)$ let
$$
\alpha_j(A) = \inf \|A-A_j\|
$$
where the infimum is over all algebraic operators $A_j$ of  degree at most $j$.
\end{definition}

Then the following holds.

\begin{proposition}  If $A$ is a bounded operator in a Banach space  such that
$$
\liminf _{j\rightarrow \infty}\alpha_j(A) ^{1/j} = 0, 
$$
then $A$ is quasialgebraic.
\end{proposition}
\begin{proof}  See Theorem  5.10.4 in [33]. 
\end{proof}


In separable Hilbert spaces $\alpha_j(A)\rightarrow 0$ if and only if $A \in \mathcal Bi\mathcal Q\mathcal T$, see below Theorem \ref{Voi}, by Voiculescu.

The sum of two algberaic operators need not be quasialgebraic,  
\begin{equation}
\mathcal A + \mathcal A \not\subset \mathcal Q\mathcal A.
\end{equation}
 \begin{example}\label{kaksnilpoa}  Let $Ae_k= e_{k+1}$ for even $k$  while $Be_k=e_{k+1}$ for odd $k$ and  with $A^2=B^2=0$  so that $A+B=S$ is the forward shift.  Then both $A,B$ are are in particular Riesz operators and we see  that the capacity is not invariant under Riesz perturbation as  the capacity of nilpotent operators vanishes while that of the unilateral shift equals 1.
 \end{example}
Perturbation with a compact operator, however, leaves the capacity invariant.


 \begin{theorem}\label{vaikkarilause} {\rm (Stirling [47])}.  Let $A\in \mathcal B(X)$ and $K\in \mathcal K(X)$, then 
 \begin{equation}
 {\rm cap}(\sigma(A+K)) =  {\rm cap}(\sigma(A)).
 \end{equation}
 \end{theorem}

Hence in particular  {\it the sum of a quasialgebraic and  a compact operator is always  quasialgebraic, }  
\begin{equation}
\mathcal Q\mathcal A + \mathcal K \subset \mathcal Q\mathcal A.
\end{equation}

Further,  
 \begin{equation} \label{ihansimppeli}
\mathcal A  + \mathcal F  \subset \mathcal A
\end{equation} 
and 
\begin{equation}\label{einiinsimppeli}
\mathcal A \mathcal A + \mathcal F  \subset \mathcal A \mathcal A. 
\end{equation} 
 
Here (\ref{ihansimppeli}) is a simple fact. Let $A\in \mathcal A$ be given with  minimal polynomial $p$, then  with $B\in \mathcal F$
$$
p(A+B) = p(A)+C=C
$$
where $C$ is of finite rank. If deg$(p)=d$ and rank$(B)=q$, then  there is a  polynomial $p_1$ of degree at most $q+1$ such that $p_1\circ p(A+B)=p_1(C)=0$.   In particular 
\begin{equation}\label{ihansimppeli2}
{\rm deg}(A+B) \le {\rm deg}(A) ( {\rm rank}(B)+1).
\end{equation}
The inclusion (\ref{einiinsimppeli}) follows from the characterization of $\mathcal A \mathcal A$ as those with meromorphic resolvents, Theorem 5.7.2 in [33] and from a general  perturbation result of operator valued meromorphic functions  by finite rank functions.     Theorem 6.1 in [35] covers the Hilbert space  using an exact identity which as such is limited to Hilbert spaces, but we derive below a quantitative perturbation bound holding in all Banach spaces.   

\begin{definition}  A vector valued function $F$:  $z\mapsto F(z)$ is called meromorphic for $|z| < R$ if   it is  holomorphic except at singularities, and  all singularities are poles:  for each singularity $z_0$  with $|z_0|<R$ there exists and integer $n_0 <\infty$ such that
$$
z \mapsto (z- z_0)^{n_0} F(z)
$$
is holomorphic near $z_0$.  The smallest such nonnegative $n_0$ is the multiplicity of the pole.
\end{definition}
In this connection it is convenient to deal the resolvents in the "Fredholm form"   $z \mapsto (1-zA)^{-1}$ rather than in  $\lambda \mapsto (\lambda-A)^{-1}= \frac{1}{\lambda}(1-\frac{1}{\lambda}A)^{-1}$. 


\begin{theorem} {\rm (Theorem 5.7.2, [33])}.  A bounded operator $A$ in a Banach space is almost algebraic if and only if 
$$
z\mapsto (1-zA)^{-1}
$$
is meromorphic for all $z\in \mathbb C$.
\end{theorem}

\begin{remark}
Some authors, including [14], [ 52],  have defined operators to be {\it meromorphic}, or, {\it of meromorphic type}, if the resolvent  $(\lambda - A)^{-1}$ is meromorphic for $\lambda\not=0$. Denoting these operators by $\mathcal M$ we thus have $\mathcal A\mathcal A = \mathcal M$.  Observe that in general for  bounded operators the function $(1-zA)^{-1}$ is analytic for $|z| <1/\rho(A)$ and there always exists a largest  $R \le \infty$,   such that it is meromorphic for $|z|<R$. Much of our discussion is independent of whether $R$ is finite or not.
\end{remark}

\subsection{Excursion into  meromorphic operator valued functions}
   
Consider now an operator valued function $F:  z\mapsto F(z) \in \mathcal B(X)$ which we assume to be meromorphic for $|z|<R \le \infty$ and for convenience, normalized as $F(0)=1$.  We denote  for $r<R$ 
\begin{equation}
m_\infty(r, F) = \frac{1}{2\pi} \int_{-\pi}^{\pi} {\rm log}^+\|F(re^{i\theta})\| d\theta,
\end{equation}
and if $\{b_j\}$ denote the poles,  listed with multiplicities, then the logarithmic average of poles  smaller than $r$ in modulus is
\begin{equation}
N_\infty(r,F) = \sum {\rm log}^+ \frac{r}{|b_j|}.
\end{equation}
Finally, the  tool to measure the growth of $F$ as a meromorphic function is the sum of these:
\begin{equation}
T_\infty(r, F)= m_\infty(r, F)+N_\infty(r,F).
\end{equation}
Now $A\in \mathcal A\mathcal A$ if and only if $
T_\infty (r, (1-zA)^{-1})<\infty$ for all $r<\infty$  while $A\in \mathcal A$ if and only if the resolvent is rational which  happens when
$$
T_\infty (r, (1-zA)^{-1})= \mathcal O ({\rm log} \ r) \  {\rm as} \ r \rightarrow \infty.
$$
(Corollary 3.1, [35]).  We shall first refer the  key properties of the low rank perturabation theory on [34], [35].  In order to deal with  finite rank perturbations we need  to measure not only the norm of the function but all its singular values and thus we restrict the discussion for a while to  $A \in \mathcal B(H)$.  Denote
$$
\sigma_j(A)= {\rm inf}_{{\rm rank}(B)<j}\|A-B\|
$$
and set
\begin{equation}
s(A)= \sum {\rm log^+} \sigma_j(A),
\end{equation}  
which we call the {\it total logarithmic size of $A$}, [34, 35].
Without going into details, we  replace in the definition of $T_\infty$   the  ${\rm log}^+$ of the norm by  the total logarithmic size and obtain another growth function $T_1(r, F)$ for which always $T_\infty \le T_1$.  However, with functions $F$ of the form $z \mapsto 1 + G(z)$ where $G(z)$ is of finite rank and $G(0)=0$ we have an exact inversion formula
\begin{equation}\label{invident}
T_1(r,F) = T_1(r, F^{-1})
\end{equation}
which allows us to do perturbation theory with low rank operators.   In fact, assume that $A\in \mathcal B(H)$ is such that  $(1-zA)^{-1}$ is meromorphic for $|z| <R \le \infty$ and $B$ is an operator of rank $q$.  Then we estimate as follows
$$
T_{\infty}(r, (1-z(A+B))^{-1}) \le T_{\infty}(r, (1-zA)^{-1})+T_{\infty}(r, (1-(1-zA)^{-1}zB)^{-1})
$$
using the fact that $T_\infty $ is submultiplicative.   Then we replace  in the second term on the right $T_\infty$ by $T_1$ and use the inversion identity (\ref{invident}) to get
$$
T_{\infty}(r, (1-(1-zA)^{-1}zB)^{-1}) \le T_1(r, (1-(1-zA)^{-1}zB)).
$$
In order to bound  the term now on the right hand side, notice that if $a, b$ are   positive, then
$$
{\rm log} (1+ ab ) \le {\rm log}^+ a + {\rm log}^+ b + {\rm log}\ 2.
$$
For $j>q$  we have $\sigma_j (1-(1-zA)^{-1}zB) \le 1$  while for $j \le q$ we estimate
$$
 \sigma_j (1-(1-zA)^{-1}zB) \le 1 + || (1-zA)^{-1}zB||
 $$
 which gives
 $$
 s(1-(1-zA)^{-1}zB) \le q \  {\rm log}^+ || (1-zA)^{-1}|| + {\rm log}^+ || z B|| +
 {\rm log}\ 2). 
 $$ 
Combining the estimates  we may formulate a special case of Theorem 4.1 of [34], or Theorem 6.1 [35].

 \begin{theorem}  Let $A\in \mathcal B(H)$ be such that $(1-zA)^{-1}$ is meromorphic for $|z| <R\le \infty$.  Let $B$ be a finite rank perturbation of $A$.  Then $(1-z(A+B))^{-1}$ is  also meromorphic for $|z|<R$ and the following quantitative estimate holds for $r<R$
\begin{align}
&T_{\infty}(r, (1-(1-zA)^{-1}zB)^{-1}) \notag  \\
&\le ({\rm rank}(B)+1)  T_\infty (r, (1-zA)^{-1}) + {\rm rank}(B) ({\rm log}^+ ||rB|| +
{\rm log} \ 2).\notag
\end{align}
\end{theorem}

\begin{example} Let $I$ denote the identity in $\ell_2(\mathbb N)$ and $K= {\rm diag}\{\alpha_j\}$ where $0\le \alpha_j <1$, and    $\alpha_j\rightarrow 0$.  Then $K$ is compact and we can view $I-K$ as  a compact perturbation of the identity.  Clearly $z \mapsto (I-z I)^{-1}$ is meromorphic in the whole plane and
$$
T_\infty(r, (I-z I)^{-1}) = \log^+ r 
$$
while
$$
T_\infty(r, (I-z (I - K)^{-1}))= \sum \log^+ (|1-\alpha_j| r) + \mathcal O(1).
$$
Thus,  if $K$  is  of finite rank, then $(I-zK)^{-1}$ is meromorphic  in the whole plane, otherwise it is meromorphic only for $r < 1$.  Hence, in general we have
\begin{equation}\label{alkjakompei}
\mathcal A + \mathcal K \not\subset \mathcal A \mathcal A.
\end{equation}

\end{example}

\begin{remark}Notice that this  proves (\ref{einiinsimppeli}) in the case of Hilbert space operators:    if $A\in \mathcal A \mathcal A$, then $T_\infty (r, (1-zA)^{-1})<\infty$ for all $r<\infty$ and hence, the same holds for $T_{\infty}(r, (1-(1-zA)^{-1}zB)^{-1})  $. 
\end{remark}

\begin{remark}
When $A\in \mathcal A$ with ${\rm deg}(A)=d$ then $T_\infty (r, (1-zA)^{-1} = d \ {\rm log}^+ r +  \mathcal O (1)$ so the estimation above would yield
${\rm deg}(A+B) \le ({\rm rank}(B)+1) {\rm deg}(A) + {\rm rank} (B)$. However, splitting  in estimating $||(1-zA)^{-1}zB||$ as $||z(1-zA)^{-1}|| \ || B||$ then implies (\ref{ihansimppeli2}), since $z(1-zA)^{-1}$ is a rational function of degree $d$, as is easy to verify.
\end{remark} 

We shall now  consider  $T_\infty(r, (1-z(A+B))^{-1})$ where $A, B$ are almost algebraic and of finite rank, respectively, but in a general Banach space $X$.
In Banach spaces we  need to work without the inversion  identity,  but  we can still derive a bound which implies  $\mathcal A \mathcal A + \mathcal F \subset \mathcal A\mathcal A$.   Observe that we can  decompose a finite rank operator as a sum of rank-1 operators and hence it suffices to  show that a rank-1 perturbation of an almost algebraic operator stays almost algebraic, which follows from the next  theorem with $R=\infty$.  We denote   the dual of  $X$ by $X^*$ and the functionals by $b^* \in X^*$.  When needed,  we write $<x,b^*>$ for the dual pair.    Thus, the rank-1 operator  $ab^*$ in a Banach space   maps $x \mapsto \ <x,b^*> a$.    

\begin{theorem} 
Let $X$ be a Banach space and $A\in \mathcal B(X)$ be such that $(1-zA)^{-1}$ is meromorphic for $|z| <R\le \infty$.  Let $a\in X$ and $b^* \in X^*$ be given.  Then also  $(1-z(A+a b^*))^{-1}$ is meromorphic for $|z|<R$ and the following estimate holds 
$$
T_\infty(r,(1-z(A+ab^*))^{-1} )\le 2 \ \big( T_\infty(r,(1-zA)^{-1})
+\ \log^+ r + \log^+(||a|| ||b^*||) + \log 2\big).
$$
 
\end{theorem}
\begin{proof}
Denoting $B=ab^*$   we can begin in the same way by estimating
$$
T_{\infty}(r, (1-z(A+B))^{-1}) \le T_{\infty}(r, (1-zA)^{-1})+T_{\infty}(r, (1-(1-zA)^{-1}zB)^{-1})
$$
but now we need to bound the second term on the right without the inversion  identity.  To that end notice that $(1-zA)^{-1}z ab^*$ is   a rank -1 
operator and we  may write
$$
(1-(1- zA)^{-1}z ab^*)^{-1} = 1+ \frac{z}{\varphi(z)} ab^*
$$ where
$\varphi(z)= 1- z b^*(1-zA)^{-1} a$ is a meromorphic scalar valued function such that $\varphi(0)=1.$  In particular we have
\begin{equation}\label{skalaari}
T(r, 1/\varphi) = T(r, \varphi) \le  T(r, zb^*(1-zA)^{-1}a) + \log 2.
\end{equation}
But $|zb^*(1-zA)^{-1}a| \le \| (1-zA)^{-1} \| \ |z| \  \|a\| \|b \|
$
and we hence have
\begin{equation}
T(r, zb^*(1-zA)^{-1}a) \le T_\infty(r, (1-zA)^{-1} ) + \log^+ r + \log^+ (\| a \| \|b \|).
\end{equation}
Combining we have
\begin{align}
&T_\infty, (1+ \frac{z}{\varphi} a b^*)) \notag \\
 \le & T_\infty(r,\frac{1}{\varphi}) + T_\infty (r, zab^*) + \log 2 \notag \\
 \le & T_\infty(r, (1-zA)^{-1} ) + \log^+ r + \log^+ (\| a \| \|b \| + \log 2 \notag  \\
 + & \log^+ r + \log^+ \|ab^*\| + \log 2\notag
\end{align}
and so
$$
T_\infty(r,(1-z(A+ab^*)^{-1})) \le 2 \big( T_\infty(r,(1-zA)^{-1})
+\log^+ r + \log^+(\| a \| \|b \| + \log 2\big).
$$

\end{proof}

\subsection{Compact and Riesz operators}

We collect first some properties and characterizations of Riesz operators.   

\begin{definition} A bounded operator $R$ in a Banach space $X$ is called a Riesz operator if   every nonzero spectral point is a pole  with a finite dimensional  invariant subspace.  We denote the  set of Riesz operators  in $X$ by $\mathcal R(X)$.
\end{definition}

The following characterization holds.
 
 \begin{theorem}\label{ruston} {\rm (Ruston [46]) } A bounded operator  $R$ in a Banach space is a Riesz operator if and only if 
 \begin{equation}
 \lim_{n\rightarrow\infty}\inf  \| R^n -K\|^{1/n} =0
 \end{equation}
 where the infimum is taken over all compact operators $K  \in \mathcal K.$
 \end{theorem}

In separable Hilbert spaces Riesz operators  are sums of compact and quasinilpotent ones.  This is due to West [51] and here is a minor sharpening:
 
\begin{proposition} {\rm  [10]}  Let $R$ be a Riesz operator in a separable infinite dimensional Hilbert space.  Then there exist a compact $K$ and a quasinilpotent $Q$ such that $R=K+Q$ such that  the commutator $[K,Q]$  is quasinilpotent as well.
\end{proposition}
 
 \begin{remark}
Suppose all nonzero poles  $\lambda_j$ of the resolvent of an almost algebraic operator $A$ are {\it simple }  so that $ (\lambda-\lambda_j)(\lambda-A)^{-1}$ is holomorphic near $\lambda_j$. If $P_j$ denotes the spectral projection onto the invariant subspace related to  the eigenvalue $\lambda_j$ then  one can set
$$
B=  \sum_{j=1}^\infty  \lambda_j P_j
$$
for which the resolvent  can be written, since $P_i P_j= \delta_{ij}P_i$, as
$$
(\lambda-B)^{-1} = \frac{1}{\lambda} + \sum_{j=1}^\infty  [\frac{1}{\lambda - \lambda_j} - \frac{1}{\lambda}] P_j.
$$
In particular $\sigma(A)= \sigma(B)$.  A. E. Taylor  discusses this and shows that 
if $C=A-B$ then $C$ is quasinilpotent and commutes with $B$, [49].  In [14] these are generalized  to allow  higher order poles.  We shall not use these in the following and leave the details out.  In dealing with Mittag-Leffler type expressions  a challenge is to obtain knowledge on the size of  projections $P_j$ and the sensitivity of the expression under perturbations. 
\end{remark}

We shall consider  splittings of almost algebraic operators into sums  of two operators where the first one is algberaic and the second one is small in the norm so that the  spectrum  is  near origin.  The possibility of splitting is based on the following result.  In order to state it, put for $\theta>1$
\begin{equation}\label{theettavakio}
C(\theta) = \frac{\sqrt \theta +1}{\sqrt \theta -1} + \log \frac{4e \sqrt \theta (\sqrt\theta +1)}{\sqrt \theta -1}.
\end{equation}
For example, $C(4) < 7.2$.   The following is Corollary 7.5 in [35].
Denote by $n(\rho,A)$  the  number of  poles of the resolvent larger than $\rho$ in absolute value and counted with multiplicities.  

\begin{theorem} Assume that $A\in \mathcal B(X)$  is such that $(1-zA)^{-1}$ is meromorphic for $|z| < R \le \infty$. Let $\theta>1$ be fixed.     Then for any $r>0$  such that $\theta r <R$, there exists $\rho$  such that
$$
\frac{1}{r} \le \rho \le  \frac{\sqrt \theta}{r}
$$
so that  
\begin{equation}\label{projektionestimaatti}
\log ||P_\rho || \le C(\theta) \  T_\infty(\theta r, (1-zA)^{-1})
\end{equation}
holds, where
$$
P_\rho \ = \  \frac{1}{2\pi i} \int_{|\lambda|=\rho} (\lambda-A)^{-1} d\lambda.
$$
Moreover $n(\rho,A)$  satisfies
\begin{equation}\label{asteenestimaatti}
n(\rho, A) <  \frac{1}{\log \theta} T_\infty(\theta r, (1-zA)^{-1}).
\end{equation}
\end{theorem}

Thus,  there is  a radius $\rho$ whose exact value remains  unknown,  but  it is known to be within be an interval, such that based on the growth function we get a quantitative bound for the projection and for the degree  of the algebraic part.   Notice that as the growth function $T_\infty $ of the resolvent is robust under low rank updates of the operator,  the bound holds essentially unchanged - although the radius may have moved within the interval.

With $R=\infty$ we can formulate the following corollary.


\begin{corollary}\label{projektio}
Let  $A\in \mathcal A \mathcal A$  and  choose $\varepsilon>0$ and $\theta >1$.  Then  there exists $\rho$, satisfying $\varepsilon \le \rho \le  \sqrt \theta \   \varepsilon$, such that  with 
$B=(1-P_\rho)A$ and $E=P_\rho A$  we have $A=B+E$ where  $B$ is algebraic and $E$ such that $\sigma(E) \subset \{ \lambda  :   |\lambda | < \rho \}$.  Here $P_\rho$ can be  bounded with $r= 1/\varepsilon$ by (\ref{projektionestimaatti})    , while   the degree of $B$  satisfies ${\rm deg} (B) = n(\rho,A)$ and can be bounded by (\ref{asteenestimaatti}).

\end{corollary}

\begin{example} 
We recall  Example 1.5 in [35].  Denote by $V^2$  the quasinilpotent solution operator in $L_2[0,1]$ 
solving $u'' = f$ with initial conditions $u(0)=1$, $u'(0) = 0$: 
$$
V^2 f(t)=\int_0^t (t-s) f(s) ds.
$$
Denoting further by $B$ the solution operator of the same equation with boundary conditions $u(0)=u(1)=0$ we have a negative selfadjoint operator with eigenvalues
$\lambda_j = -1/(\pi j)^2.$   Now  $V^2$ is a rank-1 perturbation of $B$.  If we denote by $A=\alpha B + (1-\alpha)V^2$  their resolvents  grow   with speed
$$
T_\infty (r, (1-zA)^{-1}) \sim \sqrt r
$$
as $r\rightarrow \infty$.
With $V^2$  all growth is seen thru $m_\infty$ while with $B$ all growth is in $N_\infty$ , the sum of $m_\infty$ and $N_\infty$ staying essentially constant along the homotopy.
Thus  for all operators along the homotopy
$$
\|P_\rho\|  \le  e^{\mathcal O(1/ \sqrt \rho)}
$$
and the maximum  number  of eigenvalues larger than $\rho$ is bounded by  $\mathcal O(1/ \sqrt \rho)$ which   in the  self-adjoint case is obvious.   Near the self-adjoint end the projections are nearly orthogonal but as  the operators  become increasingly "nonnormal" when approaching $V^2$  the  norms of the projections grow fast and the eigenvalues get compressed near the origin.
\end{example}


Consider  next the chain (\ref{ekainkluusioketju}) with $ \mathcal K \subset \mathcal R$ replacing $\mathcal A$:
\begin{equation}\label{kakkosketju}
\mathcal F \subset \mathcal K \subset   \mathcal R \subset \mathcal A \mathcal A \subset \mathcal Q\mathcal A \subset \mathcal B.
\end{equation}

\begin{proposition}  All inclusions in (\ref{kakkosketju}) are proper in infinite dimensional separable Hilbert spaces.
\end{proposition}

\begin{proof} First of all, if $A$ is any bounded invertible algebraic operator, then it cannot be a Riesz operator  which implies the third inclusion to be proper. The first inclusion is trivially proper:  let
 $H = \ell_2(\mathbb N)$ and  put $K:  e_j \mapsto \frac{1}{j}e_j$.  Then clearly $K\in \mathcal K \setminus \mathcal F$.  Let further $N: e_{2j-1} \mapsto e_{2j}$ while $e_{2j}\mapsto 0, $ so that $N^2=0$.  Then $R=K+N$ is a Riesz operator but not compact.
It is a Riesz operator  by the previous theorem as $R^2$ is compact.  However,  $R$ is not compact  because the sequence
$$
R e_{2j-1} = \frac{1}{2j-1} e_{2j-1} + e_{2j}
$$
does not contain any convergent subsequence.  Notice additionally that $KN-NK\not=0$.  In [22] the authors show additionally that this $R$ cannot be decomposed into a sum of compact $C$ and quasinilpotent $Q$ so that they would commute.  It has been shown later, that any Riesz operator can be represented as a sum of compact $C$ and  quasinilpotent $Q$ such that  the commutator $CQ-QC$ is also quasinilpotent, [10].    
\end{proof} 

We may now continue  listing the perturbation inclusions.   
 Along this chain we have clearly
 $$
 \mathcal F +  \mathcal F \subset  \mathcal F  \ ,  \    \mathcal K +  \mathcal K \subset  \mathcal K  \ \text{ and }   \mathcal R +  \mathcal K \subset  \mathcal R
 $$ while the Example \ref{kaksnilpoa}   shows that
 $$
  \mathcal R  +  \mathcal R \not \subset  \mathcal Q\mathcal A.
  $$
On the other hand,  Theorem \ref{vaikkarilause}  implies that 
$$
\mathcal Q\mathcal A + \mathcal K \subset \mathcal Q\mathcal A.
$$
Here   $\mathcal Q\mathcal A $  cannot be replaced by $\mathcal A\mathcal A $, as $\mathcal A\mathcal A $ is no longer invariant under compact perturbations. In fact 
$
1+ \mathcal K \not \subset \mathcal A\mathcal A. 
$
Likewise,  in $\mathcal A \mathcal A + \mathcal F \subset \mathcal A\mathcal A
$  we cannot replace $\mathcal F$ by $\mathcal K$.  To see these,  consider $1+K$ with $K:  e_j \mapsto \frac{1}{j} e_j$ which is compact but the resolvent of $1+K$ has poles accumulating at $\lambda=1$ and  so, $1+K \notin \mathcal A\mathcal A$.  Note however that $1+K$ is polynomially compact and  as such  quasialgebraic.  We shall  concentrate in polynomial classes in then next section.


 \subsection{Observations on the polynomial classes}

 To begin with, let us observe the  following simple relations.  When we write inclusions, it means that  the inclusions hold   for operators in all Banach spaces.  But when we write $\mathcal X \not\subset \mathcal Y$ it means that there exists a Banach space  (e.g. $\ell_2(\mathbb N)$) where this happens.

\begin{proposition} We have 
$\mathcal P \mathcal F= \mathcal P \mathcal A= \mathcal A$  and
$ \mathcal P \mathcal Q \mathcal A= \mathcal Q \mathcal A$  .\end{proposition}
\begin{proof}   These follow immediately from the definitions.

\end{proof}
\begin{proposition}
We have
$\mathcal P \mathcal A \mathcal A \subset \mathcal Q \mathcal A  \not\subset \mathcal P \mathcal A \mathcal A$.
\end{proposition}

\begin{proof}  If $A \in \mathcal P \mathcal A \mathcal A$, then  the spectrum $\sigma(A)$  has only a finite number of accumulation points and  ${\rm cap} ( \sigma(A)) =0$ so that $A \in \mathcal Q \mathcal A$.  (For $\mathcal P \mathcal A \mathcal A$ see   the characterisation in Theorem 2.34  below).  To see that the inclusion is proper, take a bounded  sequence $\{Q_j\}$ of quasinilpotent operators which are not nilpotent e.g. in  $\ell_2(\mathbb N)$ and let $A$ to be 
the direct sum of $\frac{1}{j}+Q_j$.  Then $\sigma(A)=\{1/j\}_{j=1}^{\infty}\cup \{0\}$ and $A$ is quasialgebraic but for any nontrivial polynomial   $p$ the resolvent of $p(A)$ has singularities outside origin which are not poles.

\end{proof}

\begin{proposition}
We  have
$\mathcal P \mathcal F \not\subset  \mathcal F$,   $\mathcal P \mathcal  K \not\subset \mathcal K$,  $\mathcal P \mathcal R \not\subset  \mathcal R$,  and $\mathcal P \mathcal A\mathcal A\not\subset  \mathcal A\mathcal A$.
\end{proposition}
\begin{proof}
Let 
$$
M= \begin{pmatrix} 0& S\\
&0\end{pmatrix}
$$
denote the operator in $\ell_2 \oplus\ell_2$ where $ S$ denotes the forward shift.
Then $M$ is not compact but as $M^2=0$ the two first claims follow.  As $M$ is nilpotent, it is a Riesz operator  but if we consider $1+M$ then it  is not Riesz but clearly polynomially Riesz with polynomial $p(\lambda)= \lambda-1$, which simultaneously  implies the  claim on almost algebraic operators.
\end{proof} 

We can still add a few relations, not covered by the previous ones.

\begin{proposition}  We have  in all Banach spaces
\begin{equation} \label{paaf}
\mathcal P\mathcal A\mathcal A  + \mathcal F \subset \mathcal P\mathcal A\mathcal A , \ \ 
\mathcal P\mathcal K + \mathcal K \subset \mathcal K\ , \ \ \mathcal P\mathcal R + \mathcal K \subset \mathcal P \mathcal R,  \end{equation}\label{rpk}
while  there exist an infinite dimensional separable Hilbert space such that
\begin{equation} \label{eitottakaan}
\mathcal R \not\subset \mathcal P\mathcal K, \   \ \mathcal A \mathcal A + \mathcal K \not\subset \mathcal P \mathcal A \mathcal A,  \  {\rm and } \ \ \mathcal A \mathcal A  \not\subset \mathcal P\mathcal R.
\end{equation}

\end{proposition}
\begin{proof}
Let $p$ be such that $p(A)$ is almost algebraic and $F$ of finite rank. Then
$p(A+F)=p(A) + G$ where $G$ is of finite rank and (\ref{paaf}) follows from (\ref{einiinsimppeli}).  The two  other claims in (\ref{paaf}) follow  similarly.

 Let now $A_j=\frac{1}{j}$ denote  the scalar multiplication in $\ell_2$ and let $A=\oplus A_j$. Then  all nonzero spectral points corresponds to poles of the resolvent and $A\in \mathcal A\mathcal A$. On the other hand,  with any nontrivial polynomial $p$ there are eigenvalues $1/j$ such that $p(1/j)\not=0$ and  for such the invariant subspace is  infinite dimensional. Hence $A\notin \mathcal P\mathcal R$.
 
Consider now the middle claim of (\ref{eitottakaan}).  We shall define two diagonal operators $A\in \mathcal A\mathcal A$ and $K\in \mathcal K$  such that for all    integers $m\ge 0 $   the point $1/{(m+1)}$ is an accumulation point of  $\sigma(A+K)$.  Thus, for any  nontrivial polynomial $p$  the spectrum of $p(A+K)$ contains an infinite amount of accumulation points and hence $A+K$  is  not in $\mathcal P \mathcal A \mathcal A$.  In fact, for $K$ we choose simply the diagonal operator $K$ mapping $e_j \mapsto \frac{1}{j} e_j$.
In order to define $A$, notice that every positive integer $j$ can uniquely be given by a pair $(m,n)$ in the form
$$
j=2^m (1+2n).
$$
We  denote the  $j^{th}$ coordinate vector  by $e_{(m,n)}$ and can define another, noncompact, almost algebraic  diagonal operator by
$$
A: e_{(m,n)} \mapsto \frac{1}{m+1} \ e_{(m,n)}.
$$
Hence $A+K$ is again a diagonal operator such that
$$
A+K: e_{(m,n)} \mapsto [\frac{1}{m+1} + \frac{1}{2^m (1+2n)} ] \ e_{(m,n)},
$$
which shows that $1/(m+1)$ is an accumulation point  of eigenvalues for every $m\ge 0$, completing the example.
 
 Finally,  $\mathcal R \not\subset \mathcal P\mathcal K$ is due to [26].  A Riesz operator in a separable Hilbert space is polynomially compact  if  and only if  there is a positive integer $n$ such that $R^n$ is compact, Lemma 2 [26].  Foias and Pearcy [19], [43]  have   an example of a quasinilpotent operator $T$  in $\ell_2(\mathbb N)$ such that $T^n$ is not compact for any $n\ge1$.   The operator is  a weighted backward shift $ Te_1=0$ while $Te_{j+1}=\omega_j e_j$ where
the weight sequence $\{\omega_j\}$ begins
$$
\{2^{-1}, 2^{-4},2^{-1}, 2^{-16}, 2^{-1}, 2^{-4}, 2^{-1}, 2^{-64}, 2^{-1}, 2^{-4}, \cdots\}.
$$

\end{proof}

\begin{remark}  Recall that in $X_{AH}$, the infinite dimensional Banach space constructed by Argyros and Haydon [8],  all bounded operators are sums of  a multiple of identity plus a compact one.  Hence  $\mathcal P \mathcal K =\mathcal B$ and only the inclusions  $\mathcal F \subset \mathcal A \subset \mathcal P \mathcal K$ and  $\mathcal F \subset \mathcal K \subset \mathcal P \mathcal K$   are proper.

\end{remark}

We  begin with a  structure theorem for polynomially almost algberaic operators, as polynomially Riesz and polynomially compact   ones are sublasses of  these.

\begin{theorem} \label{AAstructuuri}A bounded operator $A$ in a Banach space $X$ is polynomially almost algebraic if and only if there exists a decomposition 
$X=X_0 \oplus \cdots \oplus X_d$ such that $AX_i \subset X_i$ and denoting by $A_i$ the restriction of $A$ to $X_i$, $A_0$  is algebraic while for $i=1,\dots, d$,  there exist points $\lambda_i$ such that each $A_i -\lambda_i$ is almost algebraic in $X_i$.
\end{theorem}
\begin{proof}
Assume $p(A)$ is almost algebraic and denote by $\lambda_1, \cdots, \lambda_d$ the roots of $p(\lambda)=0.$  We need to show that there exist  invariant subspaces  $X_i$ for $A$ such that each restriction $A_i - \lambda_i$ is almost algebraic in $X_i$.  

Thus, $(w-p(A))^{-1}$ is meromorphic  for $w\not=0$. Thus there exists a small neighbourhood $V$ of $0$ such that $\partial V$ does not contain any singularities of $ (w-p(A))^{-1}$ and such that $p^{-1}(\sigma(p(A)))$
splits into different components $D_j$, each containing one root $\lambda_j$ of $p$. We may also assume that the closures $\overline D_j$ do not intersect.  Then put $D_0= \mathbb C \setminus \bigcup \overline D_j$.  Let    $P_j$ to denote, for $j=0, \dots, d$, the spectral projection of $A$ wrt the spectrum inside $D_j$ and put $X_j=P_jX$.  Then  $A$ restricted to $X_0$ is algebraic, as $D_0$ contains only a finite number of poles while, when restricted to $X_j$ for $j\not=0$,  all singularities are inside $D_j$ with $\lambda_j$ as the only possible accumulation point.  Thus $A-\lambda_j$ restricted to $X_j$ is almost algebraic.  

In the other  direction, by assumption,  all singularities of $(z-A)^{-1}$  are poles in $\mathbb C \setminus \{\lambda_1, \cdots, \lambda_d\}$.  Consider $ w\mapsto (w-p(A))^{-1}$ where $p(z) = \prod (z-\lambda_j)$.
We should conclude that all nonzero singularities are poles.
This can be based on the Cauchy  integral
\begin{equation}
(w-p(A))^{-1}= \frac{1}{2\pi i}\int_\gamma \frac{1}{w-p(\lambda)} (\lambda -A)^{-1} d\lambda,
\end{equation}
where $\gamma$ surrounds the spectrum of $A$ and $w\not=p(\lambda)$ for all $\lambda$ inside and on $\gamma$.  In fact, fix a nonzero $w_0$ in $\sigma(p(A))$. 
Now follow the proof of Theorem 5.9.2, [33].

\end{proof}

\begin{remark}  This is essentially a reformulation of Theorem 5.9.2 in [33]. In [16]   this has appeared as well.

\end{remark}

The structure theorems for polynomially compact and polynomially Riesz operators can be formulated  by specifying in Theorem \ref{AAstructuuri} the operators $A_i -\lambda_i$ as compact and Riesz, respectively.  The structure theorem for polynomially compact operators is due to Gilfeather [20].  Similar  result for Riesz operators in Banach spaces is contained in [52].

Here is the original formulation  of Gilfeather.

\begin{theorem} {\rm (Gilfeather [20] )}  Let $A$ be a polynomially compact operator with minimal polynomial $p(z) = (z- \lambda_1)^{n_1} \cdots (z- \lambda_k)^{n_k}$. Then the Banach space $X$ is decomposed into the direct sum $X = X_1 \oplus \cdots \oplus X_k$ and $A= A_1 \oplus \cdots \oplus A_k$  where $A_i$  is the restriction of $X$ to $X_i$.  The operators $(A_i -\lambda_i)^{n_i}$ are all compact. The spectrum of $A$  consists of countably many points with $\{\lambda_1, \dots, \lambda_k\}$ as the only possible limit points and such that all but possibly $\{\lambda_1, \dots, \lambda_k\}$ are eigenvalues with finite dimensional generalized eigenspaces. Each point $\lambda_i \in \{\lambda_1, \dots, \lambda_k\}$  is either the limit of eigenvalues of $A$ or else  $A_i -\lambda_i$  is quasinilpotent with $X_i$  infinite dimensional.
\end{theorem}

If $A$ is algebraic and $K$ is compact, then $T=A+K$ is polynomially compact. 
In fact, if $p$ is the minimal polynomial of $A$, then $p(T)= p(A)+ C = C$ where $C$ is compact.    Catherine L. Olsen  [42]  showed  1971 that in separable Hilbert spaces the decomposition is always possible.  

\begin{theorem}\label{olsen} {\rm (Olsen [42])} 
 Each  polynomially compact $A$ in a separable Hilbert space is the sum of an algebraic operator plus a compact one.
\end{theorem}

  The decomposition holds in fact in all Hilbert spaces, [32], in the same way as West decomposition.

Using Corollary \ref{projektio}  we can have  quantitative bounds for decomposed parts of polynomially compact  operators. In fact, if $ B= p(A)$ is in the Schatten class $\mathcal S_p$ then 
$$
T_{\infty }( r, (1- z B)^{-1} ) \le  \frac{k+1} {p}  \| B \|_p^p  \   r^p  + k  \log (1+ r\|B\|)
$$
where $k$ is a nonnegative integer such that  $k< p \le k+1$ and $\| . \|_p$  denotes the Schatten norm,  see Theorem 6.5 in [35].    The following formulation is for all polynomially almost algebraic operators.




If  $A\in \mathcal P \mathcal A\mathcal A$  and  $p$  is the minimal polynomial such that $p(A)$ is almost algebraic,  denote by  $\lambda_j$ for $j=1, \cdots, d$  its roots.   Let then $\rho_0$ be small enough so that  each  
$$\gamma_\rho =\{\lambda \ : \ |p(\lambda)|=\rho\}$$ 
consists of $d$ components $\gamma_\rho^j$, each surrounding one root of $p$, when $\rho < \rho_0$.   Then denote the  contour of integration in defining the spectral projection $P_\rho$:
\begin{equation}\label{lemnisprojektio}
P_\rho = \frac{1}{2\pi i} \int_{\gamma_\rho} (\lambda-A)^{-1} d\lambda.
\end{equation}   
and further 
\begin{equation}\label{osaprojektio}
P_{\rho,j} = \frac{1}{2\pi i} \int_{\gamma_\rho^j} (\lambda-A)^{-1} d\lambda
\end{equation}
so that $P_\rho = \sum_{j=1}^d P_{\rho,j}.$   We conclude that, based on  the growth of $T(r, (1-w\ p(A))^{-1})$ that there exists a $\rho$ which allows us to obtain a bound for these projections.  We use  again the factorization  $p(\lambda)-p(A)
= (\lambda-A) q(\lambda,A)$ and  denote
$$
C_0 = \| \frac{1}{2\pi \rho} \int_{\gamma_\rho} \frac{q(\lambda,A) }{p(\lambda)} d\lambda \|
$$
and  
$$
C_j = \| \frac{1}{2\pi \rho} \int_{\gamma_\rho^j} \frac{q(\lambda,A) }{p(\lambda)} d\lambda \|.
$$
Then we have the following.
\begin{theorem} Let $A$ be polynomially almost algebraic and $p, \varepsilon, \theta, \rho_0$ as above.  Then there exists $\rho$ satisfying $\varepsilon \le \rho \le \sqrt \theta  \  \varepsilon$ such that $P_{\rho,j} $ and $P_\rho$ in (\ref{osaprojektio}) and (\ref{lemnisprojektio}) satisfy
$$
\| P_{\rho,j}\| \le C_j   M  \  {\text  and }  \ \| P_{\rho}\| \le C_0   M
$$
where
$$ \log M = C(\theta) T_\infty ( \theta/\varepsilon, (1-w \  p(A))^{-1}).
$$
With $X_0= (1-P_\rho)X$ and  $X_j= P_{\rho,j} X$ for $j=1,\cdots, d$  the space gets splitted into invariant subspaces, $X= X_0 \oplus \cdots \oplus X_d$, such that $A$ restricted to $X_0$ is algebraic with ${\rm deg} A_0 \le d  \  n(\rho,p(A))$   while $A-\lambda_j$ restricted to $X_j$ is almost algebraic.  Here $C(\theta)$  satisfy (\ref{theettavakio})  and 
$$
n(\rho, p(A)) < \frac{1}{\log \theta} T_\infty( \theta/\varepsilon, (1- w  \ p(A))^{-1}).
$$
\end{theorem}

\begin{proof} This follows from writing  with $w=1/p(\lambda)$ 
$$
(\lambda-A)^{-1} = q(\lambda,A) (p(\lambda)-p(A))^{-1} = \frac{q(\lambda,A)}{p(\lambda)} (1-w \ p(A))^{-1}
$$   
and  returning to the earlier discussion.  Note that the number of  poles $w_k$  of $(1- w \ p(A))^{-1}$  satisfying $|w_k | > \rho$ is given by $n(\rho,p(A))$ and $p^{-1}(w_k)$ contains at most $d$ points.

\begin{example}  {\bf Elastic Neumann-Poincar\'e operators}  give examples of "real life" polynomially compact operators.  In short, in bounded domains in dimensions 2 and 3 with smooth boundaries the operators are  polynomially compact   but not compact.  If the boundary has  corners,  continuous spectrum appears.  For details, see  Section 6  of  the survey article [1].   
\end{example}

\end{proof}

\section{Special results in Hilbert spaces}

\subsection{Special classes}

In this section  we assume  that the space is a separable Hilbert space.  The first observation to be made is that the classes $\mathcal X$  we have discussed  above are such that in Hilbert spaces $A \in \mathcal X$ typically imply $A^*\in \mathcal X$.    We start with a class where this does {\it not} hold.

\begin{definition}   A bounded operator $A$ in a separable Hilbert space is quasitriangular, if there exists an increasing sequence $\{P_n\}$ of finite rank  orthogonal projections converging pointwise to the identity such that
$$
\lim_{n\rightarrow\infty} \|(1-P_n)AP_n\| =0.
$$
We denote then $A \in \mathcal Q\mathcal T$. It is  biquasitriangular,
$A \in \mathcal Bi \mathcal Q\mathcal T$, if both $A$ and $A^*$ are quasitriangular.
\end{definition}

Quasitriangularity was introduced by Halmos in [25], where he in particular proved that
\begin{equation}
\mathcal Q \mathcal T + \mathcal K \subset \mathcal Q \mathcal T.
\end{equation}
He later proved,    that $\mathcal Q \mathcal A  \subset \mathcal Q \mathcal T$.  Since $A^*\in \mathcal Q \mathcal A $  if and only if $A\in \mathcal Q \mathcal A $  we have   in fact 
\begin{equation}
 \mathcal Q \mathcal A  \subset \mathcal B i\mathcal Q \mathcal T,
 \end{equation} 
and this is again proper.

Another  proper subclass of $\mathcal B i\mathcal Q \mathcal T$  is provided by normal operators and their perturbations.   We denote the normal operators by $\mathcal N_{orm}$ and these are in particular quasidiagonal.

\begin{definition}   A bounded operator $A$ in a separable Hilbert space is quasidiagonal, if there exists an increasing sequence $\{P_n\}$ of finite rank  orthogonal projections converging pointwise to the identity such that
$$
\lim_{n\rightarrow\infty} \|AP_n - P_n A\| =0.
$$
We denote then $A \in \mathcal Q\mathcal D$.  
\end{definition}

The following chain holds
\begin{equation}
\mathcal N_{orm} \subset \mathcal N_{orm}+\mathcal K \subset \mathcal Q\mathcal D \subset \mathcal B i\mathcal Q \mathcal T.
\end{equation}

\begin{theorem} \label{Voi} {\rm (Voiculescu [49])}

In separable Hilbert spaces the norm closure of algebraic operators equals the biquasitriangular ones: $$ {\rm cl} \ \mathcal A=\mathcal B i\mathcal Q \mathcal T .$$
\end{theorem}
Observe that the set of algebraic operators  of  uniformly bounded degree is,  on the other hand, closed.
See more e.g. [27].  

An important subclass of $\mathcal Q \mathcal T$ consists of  the so called  {\it thin} operators, that is sums,  of scalar and compact operators.  Douglas and Pearcy showed that  a bounded operator $A$ is thin iff
$$
Q(A)= \limsup_{P} \| (I-P)AP\| =0,
$$
where $P$ runs over all  increasing sequences of finite rank orthogonal projections, while  $A$ being quasitriangular  can be written  likewise as
$$
q(A)= \liminf_{P} \| (I-P)AP\| =0. 
$$
These were considered in 1970's  and classifying nonquasitriangular operators the ratio of $q(A)/Q(A)$  turned out  important;   on these developments see e.g. [18].  Observe that it is essential that one requires the sequence of projections to  be ordered.  In fact for  every bounded $A$ and any $n$ and $\varepsilon$ there exists a  rank-n orthogonal projector  $P_n$ such that $ \| (I-P_n)AP_n\| < \varepsilon$, [25]. 

To end this list of particular classes, we  still mention the following results. 

\begin{theorem} {\rm (Apostol, Voiculescu, [7])} 

 In separable Hilbert spaces  every quasinilpotent bounded operator is a norm limit of nilpotent ones: $ \mathcal Q \mathcal N \subset {\rm cl } \ \mathcal N$.
\end{theorem} 

\begin{theorem} {\rm (Apostol, Foias [5]) }

   An operator $T\in \mathcal Bi\mathcal Q \mathcal T$ if and only if it is unitarily similar to  a block operator  $\begin{pmatrix} A&C\\
D&B \end{pmatrix}$
where $A$ and $B$ are block  diagonal (direct sums of finite size operators ) and at least one of $C$ and $D$ is compact.
\end{theorem}

\begin{definition}  If $A\in \mathcal B i\mathcal Q \mathcal T$   has connected spectrum and essential spectrum with $0 \in \sigma_e(A)$, then we denote it as
$A\in \mathcal C$.
\end{definition}
This class is interesting as every operator in $\mathcal B \setminus \mathcal C $ has a nontrivial hyperinvariant subspace.

\begin{theorem} {\rm (Apostol, Foias, Voiculescu [6]) } 

In separable Hilbert spaces  we have
$ {\rm cl} \ \mathcal N=\mathcal C.
$
\end{theorem}


\subsection {Polynomially normal   operators}

In  [21]  Gilfeather   asks  what operators $A$ satisfy equations of the form $f(A)=N$ where $N$ is a normal operator.  In particular, if $f$ is a polynomial it leads to a classification of polynomially normal operators.  Kittaneh [30] discusses this further and it contains the following formulation of Gilfeather's theorem. 

\begin{theorem} Let $A$ be a bounded operator  in a separable infinite dimensional Hilbert space and assume there exists a  nontrivial polynomial $p$ such that $p(A)$ is normal.  Then there exist reducing subspaces $\{ H_n\}_{n=0}^\infty$ for $A$ such that  
$H= \bigoplus_{n=0}^\infty H_n,$  $A_0= A_{| H_0}$ is algebraic, and $A_n=A_{| H_n}$ is similar to a normal operator for $n>0$.
\end{theorem}

For further results we refer to [30]. 

\begin{remark}
In finite dimensional spaces similarity  to diagonal matrices is  an important subset of matrices for which e.g. a functional calculus  can be naturally defined pointwise:  If $A=SDS^{-1}$ with $D= {\rm diag} (d_i)$, then defining $f(D)= {\rm diag} (f(d_i))$one can set $f(A)=Sf(D)S^{-1}$.  Of course,  using  characteristic polynomials  all matrices are "polynomially diagonalisable".  However, in practise the natural question is  about a simplifying polynomial  with smallest degree such that $p(A)$ is diagonalisable.  This is easy to describe by assuming a Jordan form  of $A$ to be known.  In fact, if $J$ denotes a $k \times k$  matrix
$$
J= \begin{pmatrix}  \lambda&1& \\
 &.&. &  &\\
 & &.&. &\\ 
 &  & & \lambda &1\\
  & & & & \lambda
\end{pmatrix}
$$
then  we want  $p(J)=   p(\lambda) I$ which  requires that $p^{(\nu)}(\lambda) = 0 $ for $\nu =1, ..., k-1$. Thus, if the minimal polynomial of $A$  is
$$
m_A(z) = \prod (z - \lambda_j) ^{n_j}
$$
then  with
$$
s_A(z) = \int_0 ^z (\zeta-\lambda_j)^{n_j -1} d\zeta  + c
$$
the matrix $s_A(A)$ is diagonalisable.
 A related Banach  algebra  and functional calculus was discussed in [39].
\end{remark}

\begin{remark}
Observe that if $A$ is polynomially normal,  $p(A)^* p(A)=p(A) p(A)^*$, then   also 
$p(A)-p(0)$ is normal and  we could restrict our attention to polynomials of the form $p(z)=zq(z)$.  This is a natural requirement in particular when considering polynomially unitary operators.
\end{remark}

Normal operators $A$ have the property that  $p(A)$ is normal for all polynomials $p$. If $A$ is self-adjoint, then $p(A)$ is self-adjoint if all coefficients of $p$ are real.  This is in contrast with the unitary case as  pointed out in Lemma  \ref{vainympyra} below.   We close this by a simple example of  a self-adjoint operator which pertubed by a  diagonal one is  no longer self-adjoint but still polynomially self-adjoint.

\begin{example}
Let $S$ denote the unitary shift operator in $\ell_2(\mathbb Z)$  so that $S+S^*$ is self-adjoint.  Let $D$ denote the diagonal operator mapping $e_j  \mapsto (-1)^j e_j$  and set
$$
A= S+S^* + i D 
$$
which is normal.   Take $p(\lambda)= \lambda^2 +1$ for which $p(i D) =0$.     Then  $p(A) = S^2 + (S^*)^2$ is selfadjoint. \end{example}

\subsection{Polynomially unitary operators}

Let $A$ be an algebraic operator  and $p$  such that $p(A)=0$.  Then  $(p+1)(A)=I$ is unitary. However, it is more useful to consider polynomials  of the form $p(z)= zq(z)$, as then it is more naturally related to  e.g. solving linear equations. In fact,
consider solving $Ax=b$   and assume that there is a known polynomial $q$ such that
$Aq(A)$ is unitary.   Then  $q(A)^{-1} A^{-1} = q(A)^* A^*$ and we have  an explicit expression for the inverse:
$$
A^{-1} = q(A) q(A)^* A^*.
$$
\begin{definition}  We say that $A$ is polynomially unitary if  there exists $q$ such that
$A q(A)$ is unitary  and denote  $A \in \mathcal P_0 \mathcal U$.
\end{definition}
 
We write $\mathcal P_0$ to make the restriction that   the polynomial must vanish at origin.  Then we can formulate  the following
\begin{proposition}
An algebraic operator $A$  is  polynomially unitary if and only if it is invertible.
\end{proposition}   
\begin{proof}
Indeed, if $A q(A)$ is unitary,  $A$ must be invertible.  On the other hand, if $A$ is invertible and   algebraic, then there exists a polynomial $q$ such that $q(A) = A^{-1}$.
But then $Aq(A) = I$ is unitary.

\end{proof}
In [29] there is a discussion on how close $Aq(A)$ can be to unitary.   

\begin{proposition} {\rm (Theorem 2.5 in [29] ) }   Let $A$ be an $n \times n$ complex matrix.  Then for every eigenvalue of $A$ and for any polynomial $q$   
\begin{equation}\label {huh}
| \ |\lambda q(\lambda)|- 1 \ | \le \min_{U \in \mathcal U} \| Aq(A) - U \|| =
\max_{\|x\| =1} | \  \| Aq(A)x \| - 1 \ |.
\end{equation}
\end{proposition}
For numerical solution of  $Ax=b$  it  is then of interest to consider the minimum of the right hand side of (\ref{huh}) over all $q$ of a given degree and how fast that would decay with increasing the degree.  We refer to [29] for further discussion  on this direction.

Suppose now that $A$ is invertible (in an arbitary Banach space)  then there exists a sequence of polynomials $q_j$ such that $Q_j:= A q_j(A) \rightarrow I$ if and only if  the spectrum $\sigma(A)$ does not separate $0$ from $\infty$, see [33].   Note, however, that there are unitary operators like the  unitary shift for which  for all nonzero polynomials $q$ we have
$\|  S q(S) - I \| >1$  as   by  maximum principle $\max_{|z|!=1}  |zq(z)-1|>1$.  

We shall now consider  in more detail some examples related to  special cases  where $\sigma(A)= \mathbb T$, the unit circle.  This means that we may restrict our considerations to the polynomials $z\mapsto z^n$, which   follows from the following simple facts.

\begin{lemma} \label{vainympyra}Let $q$ be a polynomial of degree $n$ such that
 $|q(z)| =1$ on $\mathbb T$. Then $q(z)=   \alpha z^{n}$  where $|\alpha|=1$.
\end{lemma}
\begin{proof}
Write $q(z)= \alpha_n z^n + lower$ with $\alpha_n\not=0$.  Then
$$
q(e^{i\theta})\overline{q}(e^{-i\theta})= \sum_j |\alpha_j|^2 + \dots + 2 Re\{  [ \alpha_n \overline \alpha_1+ \alpha_{n-1} \overline \alpha_0] e^{i(n-1)\theta} \} + 2 Re \{ \alpha_n \overline \alpha_0 e^{i n \theta} \}.
$$
Since this has to be identically 1,  all nonconstant terms  have to vanish. Beginning from the last term we obtain $\alpha_0=0$ and then recursively $\alpha_j=0$ for all $j<n$.  Thus $q(z)= \alpha_n z^n$.
\end{proof}

 \begin{lemma}  If $q(z)= z^n + \alpha_{n-1} z^{n-1} + \cdots + \alpha_0$, then 
 $$
 |q(e^{i\theta})| \le 1
 $$
 implies $q(z) = z^n$.
 \end{lemma}
 \begin{proof} We have 
 $$
 \frac{1}{2\pi} \int_{-\pi}^\pi |q(e^{i\theta})|^2 d\theta = 1+ \sum_{j=0}^{n-1} |\alpha_j|^2.
 $$
 \end{proof}

The following theorem is  due to Sz.-Nagy [48].
\begin{theorem}
A  bounded operator is similar to unitary if and only if  it is invertible and the set $\{ A^n\}_{n\in\mathbb Z}$ is bounded.
\end{theorem}

If an operator $A$ is similar to a normal one, then   it satisfies a  Linear Resolvent Growth ({" LRG "}) condition with some constant $C$:
 
 \begin{equation} \label{LRG}
 \|(\lambda-A)^{-1} \|  \le \frac{C}{ {\rm dist} (\lambda,\sigma(A))} \ \ \ \  {\rm for }  \ \ \lambda \notin \sigma(A).
 \end{equation}

The backward shift $T$ is a contraction  satisfying LRG but is  not normal as $T^*T \not= TT^*$.  
 
 Benamara and Nikolski   have the following result. 
   
 \begin{theorem} {\rm [9]} Let $A=U+F$, where $U$ is unitary and $F$ of finite rank, be  a contraction $\|A\| \le 1$.  Then $A$ is similar to a normal operator if and only if $A$ satisfies (\ref{LRG}) and $\mathbb D \not\subset  \sigma(A).$

\end{theorem}
On the sharpness on this, see [31].   Further, Nikolski and Treil have the following.

\begin{theorem} {\rm [41]}
Let  $U$ be unitary such that its spectrum contains a nontrivial absolutely continuous part.  Then there exists a rank - 1   perturbation  $b a^*$ such that the operator $A=U+ b a^*$ satisfies (\ref{LRG}),  $\sigma(A) \subset \mathbb T$ but $A$ is not similar to a unitary operator.
\end{theorem}

We shall now go through a list of simple examples.   

Our first example deals with a diagonal  unitary operator perturbed with  a nilpotent rank-1 operator.  Depending on  whether $\varphi/2\pi$ is rational or not,  the operator is polynomially unitary. 

\begin{example}
Let
$A=  \sum_{j=1}^\infty \lambda_j e_j e_j^* + (\lambda_2-\lambda_1) e_1 e_2^*$. 
Assume then that  cl$\{\lambda_j\}= \mathbb T$  so that the spectrum of $A$ is the unit circle $\mathbb T$.
Hence, we may  consider $z\mapsto z^n$. 
We have
$$
A^n = D^n + (\lambda_2^n-\lambda_1^n) e_1 e_2^*
$$
where the off-diagonal term measures the distance from $A^n$ to be unitary.  Thus we may set e.g. $\lambda_1=1$  while $\lambda_2=e^{i\varphi}$.  
Now  we have $|\lambda_2^n-1|=0$  for some $n$ if and only if $\varphi/2\pi $ is rational.
In particular, if $\lambda_2=-1$, then $A^2$ is unitary and $A$ is similar to unitary  if we write $A= B \oplus D_3$ with 
$$ 
B= \begin{pmatrix} 1 & -2 \\ 
 &-1 \end{pmatrix}
$$ 
and $D_3  =  {\rm diag} (\lambda_3, \lambda_4, ...)$. Then  $TAT^{-1}$ is diagonal, unitary where
$$
T= \begin{pmatrix} 1&-1\\
&1\end{pmatrix} \oplus I. 
$$
Assume now that $\varphi/2\pi $ is not rational.  Then   $Aq(A)$ is not unitary for any polynomial $q$ as 
$$
\|Aq(A) (Aq(A))^* -I \| >0
$$
but still
$$
\inf_n \|A^n (A^n)^* -I\| =0.
$$ 
 \end{example}

The next example  presents an operator $A$ which is not normal  but $A^2$ is unitary.

\begin{example}
Let 
$$ 
B=\begin{pmatrix} 1&1\\
0&-1
\end{pmatrix}
$$
and with $\{e^{i \theta_j} \}$ dense on  $\mathbb T$ we set
$$
A= \bigoplus_{j=1}^\infty e^{i\theta_j}B.
$$
With $I$ denoting the 2-dimensional identity we have, since $B^2=I$,
$$
A^2= \bigoplus_{j=1}^\infty e^{2 i\theta_j}I
$$
and hence $A^2$ is unitary.  
\end{example}

\begin{example}
Let $\{\lambda_j\}$ de dense in $\mathbb D$  and if $B$ is as in the previous example, then set
$$
A= \bigoplus_{j=1}^\infty  \lambda_j B.
$$
Then again $A^2$ is normal, and $A$ has large spectrum: $\sigma(A)=  \overline {\mathbb D}$.

\end{example}

 
 \begin{example}
 
In this example \footnote{I must have seen this one somewhere but cannot  find the reference} we  meet an operator $A  \in l_2(\mathbb N)$ with the following properties

\bigskip

(i)  $A$ is similar to unitary

\bigskip

(ii)  $\|A^n\| = 2$   for all $n \not= 0$.

\bigskip

(iii)  $\sigma(A) = \mathbb T$.

\bigskip 

\noindent Denote by $C_n$ the circulant unitary matrix in $\mathbb C^n$ such that  $c_{i+1,i}=c_{1,n}=1
$  and let $D_n(r)$ be the diagonal matrix such that $d_{i,i}=r^{-1}$ ($i<n$) while $d_{nn}=r^{n-1}$.   Then 
$
A_n(r)= C_n D_n(r)
$
is given in its polar form, with eigenvalues at the roots of unities.
In particular 
$$
\|A_n(r)^k\|= r^{-k}  \text{ for }  k<n
$$
while 
$
A_n(r)^n=I.
$
Choose $r=r_n= 2^{\frac{-1}{n-1}}$ so that $\|A_n(r_n)^{n-1}\| =2$. 
Finally set
$
A= \bigoplus_{n=2}^\infty  A_n(r_n). 
$
Now this one has the properties asked for, as  we have likewise, for $k\le n$
$$
\|A_n(r)^{-k}\| =r^{k-n}
$$
and hence, $\|A^{-k}\| =2$,  $k>0$. Summarizing we have
$
\| A^n\| =2   \   \text{ for }  n \not=0.
$

\end{example}

  The next example shows an operator $T$ such that $p(T)$ is unitary.  

\begin{example}
Let $S$ denote the unitary shift in $\ell_2(\mathbb Z)$ and $p(\lambda)=(\lambda - \lambda_1)(\lambda-\lambda_2)$.  Note that we may set e.g. $\lambda_2=0$ to obtain $p$ of the form $zq(z)$. 
Define $T$ as  follows
$$
e_k \mapsto e_{k+1} + \lambda_1 e_k   \  \text{  for  } \  k   \  \text{  odd  }  
$$
$$
e_k \mapsto e_{k+1} + \lambda_2 e_k   \  \text{  for  } \  k   \  \text{  even}.
$$
Denoting by $D$ the diagonal operator with  $\lambda_1$ and $\lambda_2$ alternating on  the diagonal we have  $p(D)=0$ and $T=S+D$.  Now we have
$$
p(T)=S^2.
$$
In fact, suppose  $k$ is odd. Then
$$
T^2 e_k = T (e_{k+1} + \lambda_1 e_k  ) = e_{k+2}+ (\lambda_1+\lambda_2) e_{k+1} +\lambda_1^2 e_k
$$
$$
-(\lambda_1+\lambda_2) T e_k = -(\lambda_1+\lambda_2) e_{k+1}-(\lambda_1+\lambda_2) \lambda_1e_k
$$
and we obtain
$$
p(T)e_k = (T^2 -(\lambda_1+\lambda_2) T + \lambda_1\lambda_2) e_k=e_{k+2}.
$$
With $k$ even the computation is analogous.   Further, if $p(\lambda)-p(z)=(\lambda-z)p[\lambda,z]= (\lambda-z)(\lambda+ z - \lambda_1-\lambda_2)$ then
$$
(\lambda-T)^{-1} = p[\lambda, T] (p(\lambda)-p(T))^{-1} = p[\lambda, T] 
(\lambda^2-(\lambda_1+\lambda_2)\lambda + \lambda_1\lambda_2 - S^2)^{-1}.
$$
Notice that this holds as such if we   denote  by $S$ the   forward shift in $\ell_2(\mathbb N)$.
By spectral mapping theorem the spectrum of $T$  is the lemniscate  $\{ \lambda \ :  |p(\lambda)|=1 \}$   (or together with the inside if $S$  not invertible).

We may modify the operator $T$ as follows.
Let   $T=D+\rho S$  where $\rho >0$.   Then  the same calculation gives
$$
p(T)=  (\rho S)^2.
$$
In particular we may have  the spectrum of  $T$ to equal the lemniscate, with any level $\rho$. 

\end{example}

\begin{example}
 Now let $S$ be again the unitary shift and set $T=S+ \alpha e_0 e_0^*$.    By Weyl's  theorem $\sigma(T) \subset \sigma(S) \cup \sigma_p(T)$.   So, assume $|\lambda|\not=1$. 
 We have
 $$
 (\lambda-T)^{-1}= (1-\alpha f_0 e_0^*)^{-1} (\lambda-S)^{-1}
 $$ where
 $$
 f_0= (\lambda-S)^{-1} e_0.
 $$
Since
$$
(1-\alpha f_0 e_0^*)^{-1} = 1+ \frac{\alpha}{1- \alpha e_0^* f_0} f_0 e_0^*
$$ 
$\lambda$ is an eigenvalue iff $ \alpha e_0^*f_0 =1$  and $f_0$ is likewise an eigenvector, if
$$
(\lambda-T) f_0= e_0 - \alpha e_0 e_0^* f_0 =0.
$$
When $|\lambda|>1$ we have $$f_0 = (\lambda-S)^{-1}e_0 = \lambda^{-1} e_0 + \lambda^{-2} e_1 + \cdots
$$ and thus $ \alpha e_0^* f_0= 1$ iff $\alpha = \lambda$.   

When $|\lambda|<1$ we have 
$$
f_0=  - S^{-1}(1- \lambda S^{-1})^{-1}e_0 = - e_{-1} - \lambda e_{-2} - \cdots
$$
and thus $\alpha e_0^* f_0 = 0$.

Thus
$$
\sigma(T) = \sigma(S) \cup \{ \alpha \} 
$$
if $|\alpha | >1$,  while  otherwise $\sigma(T)=\sigma(S)$.

Notice that again the claim stays the same if we replace $S$ by the forward shift with  the unit disc as the spectrum.

\end{example}

\begin{example}
Now $T=S+ \alpha e_0 e_k^*$ with $k \ge 1$.  For $|\lambda|>1$  the condition now is $\alpha e_k^* f_0= \alpha \lambda^{-k-1}=1$  while for $|\lambda| <1$ we have $ \alpha e_k^* f_0 = 0$.  Thus
$$
\sigma(T) = \sigma(S) \cup \{ \lambda_1, \cdots, \lambda_{k+1}\}
$$
 for $|\alpha|>1$,   where $ \lambda_j$ denote the $k+1$ roots of $\lambda^{k+1} = \alpha$,  while  for $|\alpha|<1$  we have
 $$
 \sigma(T)=\sigma(S).
 $$
\end{example}

\begin{example}
Consider now $T=S+ \alpha e_0 e_{-k}^*$ with $k \ge 1$.   The condition  for $|\lambda| >1$ now reads
$$
\alpha e_{-k}^* f_0 = \alpha e_{-k}^* (\lambda^{-1} e_0 + \lambda^{-2} e_2 + \cdots )=0$$ and $(\lambda - T)$ is invertible for $|\lambda| >1$.   On the other hand , for $|\lambda| <1$ the condition takes the form 
$$
\alpha e_{-k}^* f_0 = \alpha e_{-k}^* ( -e_{-1} - \cdots - \lambda^{k-1}e_{-k}  - \cdots)= - \alpha \lambda^{k-1} =1
$$
Hence, with $k=1$ the operator $T$ is invertible except  when $\alpha= -1$ -  and then $T$ actually splits into a  sum of forward and backward  shifts  with  the open disc consisting of eigenvalues. 

For $k >1$ we obtain eigenvalues $\{ \lambda_1, \cdots, \lambda_{k-1}\}$ where $\lambda_j$ are the roots of $\lambda^{k-1} = 1/ \alpha$.  
Thus we summarize  with $k=1$:
$$
\sigma(T) = \sigma(S)  \text { for }  \alpha \not=-1   
$$
$$
\sigma(T)= \sigma(S)  \cup \mathbb D, \text { for }  \alpha = -1
$$
while for $k \ge 2$
$$
\sigma(T)= \sigma(S)   , \text { for }  |\alpha| \ge 1
$$
$$
\sigma(T)= \sigma(S)  \cup \{\lambda_1, \cdots, \lambda_{k-1}\} , \text { for }  |\alpha|<1.
$$
 \end{example}

\bigskip

\section{Block triangular  operators}

\subsection{Notation and spectrum}

Let $ X$ and $Y$ be Banach spaces.  We consider  triangular operators  of the form
\begin{equation}
M_C= \begin{pmatrix} A & C \\
&B\end{pmatrix}
\end{equation}
where $A \in \mathcal B(X)$, $B\in \mathcal B(Y)$ and $C\in \mathcal B(Y,X)$.  Thus $M_C$ is always a bounded operator in $\mathcal B(X\oplus Y)$.  When $X$ and $Y$ are Hilbert spaces  and  $(x,y) \in X\oplus Y$, we norm $\|(x,y)\|^2 = \|x\|^2 + \| y\|^2$ while in general e.g.  $\|(x,y)\| = \max \ \{\|x\|, \|y\|\}$.

To motivate  the special interest in these block operators, notice that if we write
$$
M_C= \begin{pmatrix} A &\\  & B\end{pmatrix} + \begin{pmatrix} & C \\  & \end{pmatrix} = M_0 + N
$$ 
then $M_C$ can be thought as a perturbation of the block diagonal operator by a nilpotent one, as $N^2=0$. 

Similarly to the discussion in the previous section we first  list properties  on $M_0$ which are preserved under addition of a corner element $C$.  And then  we ask - if  the original property  is not preserved, whether there would  exist a $p$ such that $p(M_C)$ would again share  the original property. 

We need  to  know the relation between the spectrum of $M_C$ and those of $A$ and $B$.   Clearly $\sigma(M_0) = \sigma(A) \cup \sigma(B)$. 

\begin{lemma}\label{spektritsamat}   We have always 
\begin{equation}\label{osanavain}
\sigma(M_C) \subset \sigma (M_0)
\end{equation}
and 
\begin{equation}
\widehat {\sigma(M_C)} = \widehat {\sigma(M_0)}.
\end{equation}
\end{lemma}
\begin{proof}
Take $\lambda \notin \sigma(M_0)$.  Then $(\lambda-M_C)$ is invertible. In fact we may  set 
$$
X=\begin{pmatrix} (\lambda-A)^{-1} & (\lambda-A)^{-1} C (\lambda-B)^{-1} \\
&  (\lambda-B)^{-1}\end{pmatrix}
$$
and multiply $X (\lambda-M_C) = 1$.  Thus we have
\begin{equation}\label{aakulmalla}
(\lambda - M_C)^{-1} = \begin{pmatrix} (\lambda-A)^{-1} & (\lambda-A)^{-1} C (\lambda-B)^{-1} \\
&  (\lambda-B)^{-1}\end{pmatrix}.
\end{equation}
Choose then an arbitrary  boundary point $\lambda_0$ of $\widehat{\sigma(M_0)}$. Then $\|(\lambda  - M_0)^{-1}\| \rightarrow \infty$ as $\lambda \rightarrow \lambda_0$ from outside of $\widehat{\sigma(M_0)}$.  Then necessarily $\|(\lambda-M_C)^{-1}\| \rightarrow \infty$ as well.  Thus, $\lambda_0$ is a also a boundary point of $\widehat{\sigma(M_C)}$. 

\end{proof}

We shall be mainly interested in classes where spectra do not have interior points.  Then    in particular $\sigma(M_C)= \sigma(A) \cup \sigma(B)$.  However,   in general the inclusion in (\ref{osanavain}) can be proper.

\begin{example}\label{outounitary} ([17]) Let $H= l^2(\mathbb N) \oplus l^2(\mathbb N)$
and $A=S$ the forward shift, $B=S^*$ the backward shift and 
$$
C:  \{\eta_j\} \mapsto \{ \eta_1, 0, 0, \dots\}.
$$
Then $M_C= \begin{pmatrix} S&C\\&S^*\end{pmatrix}$ is unitary  and  $\sigma(M_C) = \partial \mathbb D$ while $\sigma(M_0)= \overline{\mathbb D}$ 
so that
$
\sigma(M_C) \subsetneq \sigma(M_0) .
$ In particular $M_C$ is invertible 
$
M_C^{-1}= M_C^*= \begin{pmatrix} S^*&\\ C^*& S\end{pmatrix}
$.  However,  $M_C^{-1}$ is not upper block triangular and hence not in the closed subalgebra generated by the  nonnegative powers of $M_C$.
\end{example}

Notice that if   the {\it Sylvester equation} 
\begin{equation}\label{sylvester}
AX-XB=C
\end{equation} has a solution $X$, then we have
\begin{equation}
\begin{pmatrix} I&X\\
&I\end{pmatrix}\begin{pmatrix} A&C\\
&B\end{pmatrix}\begin{pmatrix} I&-X\\
&I\end{pmatrix}=\begin{pmatrix} A\\
&B\end{pmatrix}
\end{equation}
and in particular, $M_C$ and $M_0$ have the same spectrum.  Further, since

\begin{equation}
\begin{pmatrix} A&C\\
&B\end{pmatrix}=
\begin{pmatrix} I&\\
&B\end{pmatrix}\begin{pmatrix} I&C\\
&I\end{pmatrix}\begin{pmatrix} A\\
&I\end{pmatrix}.
\end{equation}
we see that if $M_C$ is invertible, then $A$ is left invertible and $B$ invertible from right.

\bigskip

   \begin{definition}
 
 The {\it approximate defect spectrum} $\sigma_{\delta}(A)$ is the set 
 $$\sigma_{\delta}(A)=\{\lambda\in \mathbb C \ : \ \lambda-A  \  \text { is not onto }  \}.
 $$
 The approximate point spectrum $\sigma_a(A)$ is the set  of $\lambda \in \sigma(A)$ for which there  exists a sequence $\{x_n\}$ of unit vectors  such that $\| Ax_n - \lambda x_n\| \rightarrow 0$.

 \end{definition}
 
 \begin{theorem} {\rm (Davis and Rosenthal [13] )} If
\begin{equation}\label{sylvesterinehto}
 \sigma_\delta(A) \cap \sigma_a(B) = \emptyset
 \end{equation}
 holds, then the Sylvester equation (\ref{sylvester})
 has a solution for every $C$.
 If $A$ and $B$ act in Hilbert spaces, then  (\ref{sylvesterinehto}) is also  necessary.
 \end{theorem}
 
In Hilbert spaces the following holds (Corollary 3.4 in [28]):

\begin{theorem}   Let $A\in \mathcal B (H_1)$,  $B\in \mathcal B (H_2)$.  Then  for every $C\in \mathcal  B(H_2,H_1)$ we have
$$
\sigma(M_C)= \sigma(M_0)
$$
if one of the  following conditions hold

(i) $ \sigma(A) \cap \sigma(B) $ has empty interior

(ii)  $A^*$ or $B$ has SVEP

\end{theorem}

\subsection{Perturbation results}

\bigskip

 \begin{proposition}
 $M_C$ is of finite rank if and only if  $A, B, C$ are of finite rank.
 \end{proposition}
 
 \begin{proof}
 This is obvious.
 \end{proof}
 We have  ${\rm rank} \ M_C \le {\rm rank}  \ A + {\rm rank} \ B + {\rm rank} \ C$.

 \begin{proposition} $M_C$ is compact if and only if all $A,B,C$ are compact.
 \end{proposition}
\begin{proof} This follows from considering bounded sequences $\{(x_n,y_n)\}$ and asking for a convergent subsequence.
\end{proof} 
\begin{proposition} $M_C$  is algebraic if and only if $A$ and $B$ are algebraic.
\end{proposition}
\begin{proof} If $p(M_C)=0$ then clearly $p(A)=0$ and $p(B)=0$.   Suppose $p(A)=0$ and $q(B)=0$.  Then $(pq)(A)=0$ and $(pq)(B)=0$ so that
$$
(pq)(M_C)= \begin{pmatrix} 0& (pq)[A,B](C)\\
&0
\end{pmatrix}. 
$$
\end{proof}
In particular ${\rm deg} \ M_C \le 2 ({\rm deg} \ A + {\rm deg} \ B)$.    
 
\begin{proposition}\label{emseeaa} $M_C$ is almost algebraic  if and only if $A$ and $B$ are almost algebraic.
\end{proposition}
 \begin{proof}
This follows from the characterization of the resolvent being meromorphic in $1/ \lambda$.   We have
\begin{equation}
(\lambda - M_C)^{-1} = \begin{pmatrix} (\lambda-A)^{-1} & (\lambda-A)^{-1} C (\lambda-B)^{-1} \\
&  (\lambda-B)^{-1}\end{pmatrix}
\end{equation}
and hence  the resolvent of $M_C$ has the same singularities as  the resolvents of $A$ and $B$  together.
  \end{proof} 
From  
$$
(I-zM_C)^{-1} = (I-zM_0)^{-1} + \begin{pmatrix}  0&  z(I-zA)^{-1}C(I-zB)^{-1}\\
& 0 
\end{pmatrix}
$$
we obtain
$$
T_{\infty} (r, (I-zM_C)^{-1}) \le 2 \ ( T_{\infty}(r, (I-zA)^{-1}) +T_{\infty}(r, (I-zB)^{-1})  
+ \log^+ r  + \mathcal O (1). 
$$

\begin{proposition}  $M_C$ is  a Riesz operator  if and only if both $A$ and $B$ are Riesz operators.
\end{proposition}

\begin{proof}    As  the resolvents of $A$ and $B$ jointly have the same poles as the resolvent of $M_C$ the necessity follows from the block structure. 

In the other direction, if
 $\lambda_0\not=0$ is a  pole of either $A$ or $B$ but not both, the claim follows again from (\ref{aakulmalla}). In fact, this is seen from example from the  spectral projection

\begin{equation}
P= \frac{1}{2\pi i}  \int_\gamma (\lambda - M_C)^{-1} d\lambda
\end{equation}
with  $\gamma= \{\lambda : |\lambda-\lambda_0|= \varepsilon\}$, $\varepsilon>0$ small   enough so that all other spectral points stay outside of $\gamma$, where  the block form makes it straightforward to see that the rank of the projection is  finite.     

Suppose  therefore that $A-\lambda_0$ and $B-\lambda_0$ both have  finite dimensional null spaces, say  dimension $m$ and $n$.  It is instructive to   decompose $A=A_m \oplus A_\infty$ and $B=B_n \oplus B_\infty$  where $A_m$ and $B_n$ are the restrictions  onto the finite dimensional invariant subspaces, respectively and $A_\infty$ and $B_\infty$ likewise operate in the complementary invariant subspaces.
Then decompose $C$ into block form
$$
C= \begin{pmatrix} C_{m,n}& C_{m,\infty}\\
C_{\infty,n}& C_{\infty,\infty}
\end{pmatrix}
$$
where $C_{m,n}$ can be thought as an $m \times n$- matrix,  while $ C_{m,\infty}$   and   $C_{\infty,n}$  are finite rank operators as well.  Writing the resolvent into block form   we have 
\begin{align*}
&(\lambda-A)^{-1} C (\lambda-B)^{-1}  \\
&= \begin{pmatrix}
(\lambda-A_m)^{-1} C_{m,n} (\lambda-B_m)^{-1} & (\lambda-A_m)^{-1} C_{m,\infty} (\lambda-B_\infty)^{-1}\\
(\lambda-A_\infty)^{-1} C_{\infty,n} (\lambda-B_n)^{-1}&(\lambda-A_\infty)^{-1} C_{\infty,\infty} (\lambda-B_\infty)^{-1}
\end{pmatrix}.
\end{align*}
Here the lower right hand corner is holomorphic near $\lambda_0$  while  the other three blocks are  finite rank valued functions and the claim follows.

\end{proof}

In Example \ref{kaksnilpoa}   the sum of two nilpotent operators sum up to the shift operator.  The structure of $M_C= M_0 + N$   prevents this type of  phenomenom to happen:  

\begin{proposition}
$M_C$ is quasinilpotent  if and only if both $A$ and $B$ are quasinilpotent.
\end{proposition}

\begin{proof}

We have  always $\sigma(M_C) \subset \sigma(A) \cup \sigma(B)$ and hence $M_C$ is quasinilpotent.  Reversely, if $M_C$ is quasinilpotent, then $M_0$ is quasinilpotent as well, by Lemma \ref {spektritsamat}.   

\end{proof}

\begin{proposition} $M_C$ is quasialgebraic if and only if $A$ and $B$ are quasialgebraic.
\end{proposition}
 \begin{proof}
Assume  first that $A$ and $B$ are quasialgebraic so that both  ${\rm cap}(\sigma(A))$ and $ {\rm cap}(\sigma(B))$ vanish. While the capacity is not in general subadditive, sets of  (logarithmic) capacity zero are polar and polar sets are countably subadditive [45].  Hence 
$$
{\rm cap}(\sigma(A) \cup \sigma(B)) =0
$$
and $M_0$ is quasialgberaic, too.  We know that $\sigma(M_C) \subset \sigma(M_0)$ and hence $M_C$ is quasialgebraic as the capacity is a monotonous set function.
On the other hand,  if $M_C$ is quasialgberaic, then its spectrum  is totally disconnect and $\widehat{\sigma(M_C)} = \sigma(M_C)$.
 By Lemma \ref{spektritsamat} we obtain
 $$
   {\rm cap} \big(\widehat {\sigma(A)} \cup \widehat {\sigma(B)}\big)=  {\rm cap} (\sigma (M_C)) = 0
   $$
   and so  both $A$ and $B$ must be quasialgebraic.

  \end{proof} 
  

In Example \ref{outounitary} we have $M_C$ unitary while $M_0$ is not normal.  

\begin{proposition}\label{mcnormaali}

If $M_C$ is normal and $C\not=0$, then $M_0$ is not normal.
\end{proposition}

\begin{proof}
If $M_C$ is normal then $C^* C +B^*B=BB^*$ and thus $M_0$ is normal only if $C^*C=0$.
\end{proof}   
 
All  normal operators are quasitriangular. In  Example \ref{outounitary}   $M_0=  S\oplus S^*$  and it is known  [25] that $S$ is not quasitriangular.  Further,  $S\oplus 0=0$ is not quasitriangular but $S\oplus M$  is, where $M$ is diagonal  with   the closed unit disc as the spectrum [44]. Whether $S\oplus S^*$ is, was a question in [25].  Using  a result  in [15] we may formulate the following:
 
 \begin{proposition}  If $A$ and $B$ are quasitriangular, then so is $M_C$.
 \end{proposition}

 \subsection{ {\bf $M_C$ } polynomially almost algebraic, compact and Riesz}
 
 \bigskip
 
\begin{proposition} 
$M_C\in \mathcal P\mathcal A \mathcal A$  if  and only if both  $A\in \mathcal P\mathcal A \mathcal A$ and $B\in \mathcal P\mathcal A \mathcal A$.
\end{proposition} 
 
\begin{proof}
If $p(M_C)$ is almost algbebraic, then by Proposition \ref{emseeaa} both $p(A)$ and $p(B)$ are almost algebraic.  
In the other direction, assume that $p(A)$ and $q(B)$ are almost algebraic.  Here we need to conclude that  both $(pq)(A)$ and $(pq)(B)$ are then polynomially almost algebraic.  Then so is $(pq)(M_C)$ again by Proposition  \ref{emseeaa}.
Consider $(pq)(A)$ as $(pq)(B)$ is similar.  As $p(A)$ is polynomially almost algebraic, it means that $(\lambda-A)^{-1}$ is meromorphic except at a finite  set of points $\lambda_1, \dots, \lambda_m$ and $p$ vanishes at these points.  But then $pq$ is another  nontrivial polynomial which also vanishes at these points, and all we need to conclude that all  nonzero singularities of $(z- (pq)(A))^{-1}$ are poles.   In the proof of Theorem 5.9.2, [33] this has been carried out for the minimal polynomial  but  the discussion holds as such for any  monic polynomial   which vanish at  $\lambda_1, \dots, \lambda_m$.

\end{proof}

Consider next  polynomial compactness. While $M_C$ is compact  only if all $A,B,C$ are compact,  notice that 
$$
M_C^2 = \begin{pmatrix} A^2& AC+CB\\
& B^2 \end{pmatrix}
$$
is  compact when $A^2, B^2$ and $AC+CB$ are.  This happens in particular
when $A$ and $B$ are compact.  
This allows us to  formulate

\begin{proposition}
$M_C\in \mathcal P\mathcal K$  if and only if both  $A\in \mathcal P\mathcal K$ and $B\in \mathcal P\mathcal K$.
\end{proposition}
\begin{proof}
Let $p$ and $q$ be polynomials such that $p(A)\in \mathcal K$ and $q(B)\in \mathcal K$.  Then  both $(pq)(A)$ and $(pq)(B)$ are compact.  Thus
$$
(pq)(M_C)= \begin{pmatrix} (pq)(A)& (pq)[A,B]C\\
& (pq)(B)
\end{pmatrix}
$$
has compact diagonal blocks 
and hence  $(pq)^2 (M_C)$ is compact.  

On the other  hand, if $p(M_C)$ is compact in then both diagonal blocks are compact and thus $A$ and $B$ are polynomially compact.

\end{proof}


For Riesz operators we  restrict the discussion to Hilbert spaces, where the Olsen's characterization Theorem \ref{olsen} can be used.

\begin{proposition} Let $A$ and $B$ operate in separable Hilbert spaces. Then 
$M_C \in \mathcal P\mathcal R$ if and only if both $A\in \mathcal P \mathcal R$
and $B\in \mathcal P \mathcal R$.
\end{proposition}

\begin{proof}
We may assume that  $A=G+K$ and $B=H+L$ where $G$ and $H$ are algebraic and $K$ and $L$ are compact.  Let $p$ and $q$ be such that $p(G)=q(H)=0$ so that $p(A)$ and $q(B)$ are compact.  But then $(pq)(A)$ and $(pq)(B)$ are compact  and we conclude  again that $(pq)^2(M_C)$ is compact.  By  Theorem \ref{ruston} $(pq)(M_C)$ is then Riesz.

\end{proof}

\begin{proposition}
If $M_0 \in \mathcal N_{orm}$, then $M_C \in \mathcal P\mathcal N_{orm}$ if and only if there exists a  nontrivial $p$ such that  $p(M_0) = p(M_C)$.
\end{proposition}

\begin{proof}  If $0\not= p(M_C)\in \mathcal N_{orm}$   then  by Proposition \ref{mcnormaali} $p(B)\notin \mathcal N_{orm}$  and further $p(M_0) \notin \mathcal N_{orm}$ which  contradicts $M_0\in \mathcal N_{orm}$. \end{proof}

\bigskip

\bigskip 
 
{\bf References}

\bigskip

\noindent [1] K. Ando, H. Kang, Y. Miyanishi, and M. Putinar,  Rev.Roumaine Math. Pures Appl. 66 (2021), 3-4. 545-575

\smallskip

\noindent [2]  Diana Andrei, Multicentric holomorphic calculus for n-tuples of commuting operators, Adv. Oper. Theory, Vol. 4, Number 2 (2019), 447-461

\smallskip

\noindent [3] Diana Andrei, Olavi Nevanlinna, Tiina Vesanen, Rational functions as new variables, arXiv:2104.11088 [math.CV] (April 2021)

\smallskip

\noindent [4]  Apetrei, Diana, Nevanlinna, Olavi: Multicentric calculus and the Riesz projection, Journal of Numerical Analysis and Approximation Theory. 44 (2), 2016, p. 127-145 .

\smallskip

\noindent [5]  C. Apostol and C. Foias, On the distance to biquasitriangular
operators, Rev. Roum. Math. Pures
Appl. 20 (1975) 261-265

\smallskip

\noindent [6] C. Apostol, C.Foias, D.Voiculescu, On the norm-closure of nilpotents. II, Rev. Roumaine Math. Pures Appl. 19 (1974), 549-557

\smallskip

\noindent [7]  C. Apostol, D.Voiculescu, On a problem of Halmos Rev. Roumaine
Math. Pures Appl. 19 (1974), 283-284

\smallskip

\noindent [8]   Spiros A. Argyros and Richard G. Haydon, A hereditarily indecomposable
$\mathcal L_\infty$-space that solves the scalar-plus-compact problem,
Acta Math. 206 (2011), no. 1, 1 - 54, Doi:10.1007/s11511-011-0058-y.
MR2784662 (2012e:46031)

\smallskip 

\noindent [9] N.E. Benamura and  N.K. Nikolski,    Resolvent tests for similarity to a normal operator, Proc. London. Math. Soc., 78 (1999), no.3, pp. 585 - 626.


\smallskip

\noindent [10] C. K. Chui, P. W. Smith,  J. D. Ward, A note on Riesz operators, Proc.
Amer. Math. Soc.
Vol. 60, Oct. 1976, 92-94

\smallskip

\noindent [11]  John B. Conway, Domingo A. Herrero and Bernard B. Morrel, Completing the Riesz-Dunford functional calculus, Memoirs of AMS, November 1989, Vol. 82, Number 417
 
 \smallskip

\noindent [12]  Ra\'ul Curto, Mihai Putinar, Polynomially Hyponormal Operators, Operator Theory: Advances and Applications, Vol. 207, 195-207, Springer 2010 

\smallskip 

\noindent [13]  C. Davis and P. Rosenthal, Solving linear operator equations, Canad. J. Math. XXVI

\smallskip

\noindent [14]  John Derr, Angus E.Taylor:  Operators of meromorphic type with multiple poles of the resolvent,  Pacific J. Math. 12 (1962), no. 1, 85--111

\smallskip
 
\noindent [15] R. G. Douglas, Carl Pearcy, A note on quasitriangular operators
Duke Math. J. 37(1): 177-188 (March 1970). DOI: 10.1215/S0012-7094-70-03724-5

\smallskip

\noindent [16]  B.P. Duggal, Dragan S. Djordjevi\'c, Robin E. Harte and Sne\v zana  \v C. \v Zivkovi\'c- Zlatanovi\'c,
 Polynomially meromorphic operators, Mathematical Proceedings of the Royal Irish Academy 116A (2016), 71 - 86 ; 
http://dx.doi.org/10.3318/PRIA.2016.116.07
 
\smallskip

\noindent [17]  H.K. Du and J. Pan, Perturbation of spectrums of 2$\times $2 operator matrices, Proc. Amer.Math.
Soc. 121(1994), 761-776. MR 94i:47004

\smallskip

\noindent [18]  L.A.Fialkow, A Note on Non-Quasitriangular Operators II, Indiana Math. J.  23. No.3 (1973, 213-220

\smallskip

\noindent [19] C. Foias and C. Pearcy, A model for quasinilpotent operators, Michigan Math. J. 21 (1974), 399-404

\smallskip

\noindent [20]   
 F. Gilfeather, The structure and asymptotic behavior of polynomially compact operators, Proc. Amer. Math. Soc. 25 (1970) 127-134
 
\smallskip

\noindent [21] 
 Frank Gilfeather, Operator valued roots of Abelian analytic functions, Pacific J. Math. Vol. 55, No.1, (1974) 127- 148
 
\smallskip

\noindent [22] T. A. Gillespie and T. T. West, A characterisation and two examples of Riesz operators, Glasgow Math. J. 9 (1968), 106-110.  

\smallskip

\noindent [23]  P. R. Halmos, Capacity  in Banach Algebras,
Indiana University Mathematics Journal
Vol. 20, No. 9 (March, 1971), pp. 855-863

\smallskip

\noindent [24]  P. R. Halmos, Invariant subspaces  of polynomially compact operators, Pacific J.  Math.  Vol.16, No.3, 1966, 433 - 438

\smallskip

\noindent [25]  P.  R.  Halmos, Quasitriangular operators,  Acta Sci.  Math. (Szeged) 29  (1968),  283-293. MR 38  2627.

\smallskip

\noindent [26] Young Min Han, Sang Hoon Lee,  Woo Young Lee,  On the structure of polynomially compact operators, Math. Z.  Vol. 232 
 257-263 (1999)

\smallskip

\noindent [27]  Domingo A. Herrero, Most quasitriangular operators are triangular, most biquasitriangular operators are bitriangular, J. Operator Theory, 20, (1988), 251-267

\smallskip

\noindent [28]  Junjie Huang,  Aichun Liua, Alatancang Chen: Spectra of  2 x 2 Upper Triangular Operator Matrices, Filomat 30:13 (2016), 3587 3599
DOI 10.2298/FIL1613587H, http://www.pmf.ni.ac.rs/filomat

\smallskip

\noindent [29] Marko Huhtanen, Olavi Nevanlinna, Polynomials and lemniscates of indefiniteness, Numer. Math. (2016) 133: 233 - 253,  DOI 10.1007/s00211-015-0745-2

\smallskip

 \noindent [30]  Fuad Kittaneh, On the structure of polynomially normal operators,  Bull.Austral. Math.Soc. Vol 30, (1984), 11 -18. 

 \smallskip
 
 \noindent [31]. S.Kupin and S.Treil, Linear resolvent growth of weak contraction does not imply its similarity to a normal operator, Illinois Journal of Mathematics
Volume 45, Number 1, Spring 2001,  229 - 242  

 \smallskip
 
\noindent  [32]    
Matja\"yz Konvalinka,  Integr. equ. oper. theory, Vol. 52 (2) (2005), 271-284

 \smallskip
 
\noindent [33] Olavi Nevanlinna, Convergence of Iterations for Linear Equations, Birkh\"auser, (1993)

\smallskip

\noindent [34] Olavi Nevanlinna, Growth of operator valued meromorphic functions, Ann. Acad. Sci.Fenn. Math. Vol. 25, 2000, 3-30

\smallskip

\noindent [35] O. Nevanlinna,  Meromorphic Functions and Linear Algebra,   \ AMS Fields Institute Monograph 18  (2003)

\smallskip

\noindent [36] O. Nevanlinna, Computing the spectrum and representing the resolvent, Numerical Functional Analysis and Optimization, 30 (9 - 10):1025 - 1047, 2009

 \smallskip
 
\noindent [37] O. Nevanlinna, Multicentric Holomorphic Calculus, Computational Methods and Function Theory, June 2012, Vol. 12, Issue 1, 45 - 65.

\smallskip

\noindent [38] O. Nevanlinna, Lemniscates and K-spectral sets, J. Funct. Anal. 262, (2012), 1728 - 1741.

\smallskip

\noindent [39] O. Nevanlinna, Polynomial as a New Variable - a Banach Algebra with Functional Calculus, Oper. and Matrices 10 (3) (2016)  567 - 592

\smallskip

\noindent [40] O. Nevanlinna, Sylvester equations and polynomial separation of spectra, Oper.  and Matrices 13, (3) (2019),  867- 885

\smallskip

\noindent [41] 
N.Nikolski, Sergei Treil, 
Linear resolvent growth of rank one perturbation of a unitary operator does not imply its similarity to a normal operator, 
December 2002,  Journal d'Analyse Math\'ematique 87(1):415 - 431
DOI: 10.1007/BF02868483

\smallskip
 
\noindent [42] Catherine L. Olsen
 A Structure Theorem for Polynomially Compact Operators,  American Journal of Mathematics, Vol. 93, No. 3 (Jul., 1971), pp. 686-698 
 
\smallskip
\smallskip 
\noindent [43] C.M. Pearcy, Some recent developements in operator theory, CBMS 36, Providence:AMS, 1978.

 \smallskip

\noindent  [44]  C. Pearcy, N. Salinas, Can. J. Math., Vol. XXVI, No. 1, 1974, pp. 115-120

\smallskip

\noindent [45]  Th. Ransford, Potential Theory in the Complex Plane, London Math. Soc. Student Texts 28, Cambridge Univ. Press, 1995
  
\smallskip

 \noindent [46] 	A.F. Ruston, Operators with a Fredholm theory, J. London Math. Soc.  29 (1954) pp. 318 - 326
 
 \smallskip
 
 \noindent [47]  David S.G. Stirling, The  Capacity of Elements of Banach Algebras,  Doctor of  Philosophy Thesis, University of Edinburgh,  1972
 
\smallskip
\noindent
[48] B. Sz.-Nagy, On uniformly bounded linear transformations in Hilbert Space, Acta Sci. Math., 11 (1947), 152-157.

\smallskip

\noindent [49]  A. E. Taylor, Mittag-Leffeler expansions and spectral theory,
 Pacific J. Math., 10 (1960), 1049-1066
 
 \smallskip
 
\noindent [50]    D. Voiculescu, Norm-limits of algebraic operators, Rev. Roumaine Math. Pures et Appl. 19 (1974), 371-378. 
 
 \smallskip
 
 \noindent [51]   T. T. West, The decomposition of Riesz operators, Proc. London Math. Soc. (3) 16 (1966),  737-752
 
 \smallskip

\noindent  [52]  Sne\v zana  \v C. \v Zivkovi\'c-Zlatanovi\'c,
 Dragan S. Djordjevi\'c, Robin E. Harte, Bhagwati P. Duggal,
 Filomat 28:1 (2014), 197 205
DOI 10.2298/FIL1401197Z

\bigskip

  {\bf  Notation and definitions}

 \bigskip
 
 $\mathcal A$   \ \ \  \ \ \  \ \  algebraic operators,  Def  2.6 
 
 $ \mathcal A  \mathcal A$   \ \ \ \ \ almost algebraic, Def 2.6 
  
  $ \mathcal B$   \ \ \ \ \ \   \ \ \ bounded operators
   
 $  \mathcal B i  \mathcal Q \mathcal T$  \ \  biquasitriangular, Def 3.1
 
 $ \mathcal C$   \ \ \ \ \ \   \ \ \  see Def 3.6  ($ = $ cl $ \mathcal N$)
 
  $ \mathcal F$  \ \ \   \ \ \  \ \ \ finite rank
  
  $ \mathcal K$   \ \ \  \ \ \  \ \ \ compact  
  
  $\mathcal M$  \ \ \  \ \ \  \ \  meromorphic, of meromorphic type   ($ = \mathcal A  \mathcal A$)
 
  $ \mathcal N$   \ \ \  \ \ \  \ \ \  nilpotent
  
   $ \mathcal N_{orm}$  \ \ \    normal operators
   
 $\mathcal U$    \ \ \  \ \ \  \ \ \ unitary operators
  
$  \mathcal P \mathcal A \mathcal A$  \ \ \ polynomially almost algebraic,  Def 2.1 and Def 2.6

$ \mathcal P \mathcal K$  \ \ \ \ \ \  polynomially compact

$ \mathcal P \mathcal N_{orm}$       polynomially normal

$ \mathcal P \mathcal R$  \ \ \  \ \ \ polynomially Riesz

$\mathcal R$  \ \ \  \ \ \  \ \ \  Riesz,  Def 2.21

$  \mathcal Q \mathcal A$ \ \ \  \ \ \   quasialgebraic, Def 2.6
  
 $  \mathcal Q  \mathcal D$   \ \ \ \ \ \  quasidiagonal, Def 3.2
 
 $ \mathcal Q \mathcal N$  \ \ \  \ \ \  quasinilpotent,  
  
 $  \mathcal Q \mathcal T$  \ \ \  \ \ \  quasitriangular, Def 3.1

 \bigskip

 \bigskip
 
$ \sigma(A)$  \ \ \ \ \    spectrum of $A$

$\sigma_\delta(A)$\ \ \ \  approximate defect spectrum, Def 4.3

$\sigma_a(A)$ \ \ \ \ \   approximate point spectrum, Def 4.3

$\sigma_j(A)$ \ \ \ \ \      $j^{th}$ singular value of $A$

$s(A)$\ \ \ \ \ \ \ \   total logarithmic size (2.15)

$\alpha_j(A)$ \ \ \ \     distance to algebraic operators of degree j,  Def 2.9

$\widehat K$ \ \ \ \ \ \   \ \ \ \ \  polynomially convex hull of a compact set $K$

$m_A(z)$  \ \ \ \ \ \   \  minimal polynomial of $A$

$s_A(z)$  \ \ \ \ \ \   \ \ \   simplifying polynomial  of $A$, Remark 3.9
 
 $m_\infty(r,F)$		\ \ \    			(2.12)
 
 $N_\infty(r,F)$		\ \ \    			(2.13)
 
 $T_\infty(r,F)$		\ \ \ \   			(2.14)
 
 $T_1(r,F)$		\ \ \ \ \  \			(2.16)

\end{document}